\documentclass[a4paper]{amsart}
\pdfoutput=1
\usepackage[utf8]{inputenc}
\usepackage{amssymb,pstricks,amscd}
\usepackage[T1]{fontenc}
\usepackage[totalwidth=17cm,totalheight=24cm]{geometry}
\usepackage{graphicx}
\newtheorem{cor}{Corollary}[section]
\newtheorem{thm}[cor]{Theorem}
\newtheorem{ques}{Question}
\newtheorem{Hthm}{Theorem H\hspace{-0.1cm}}
\newtheorem{Sthm}{Theorem S\hspace{-0.1cm}}
\newtheorem{prop}[cor]{Proposition}
\newtheorem{lemma}[cor]{Lemma}

\theoremstyle{definition}
\newtheorem{defi}[cor]{Definition}
\theoremstyle{remark}
\newtheorem{remark}[cor]{Remark}
\newtheorem{example}[cor]{Example}

\newcommand{\N}{{\mathbb N}}
\newcommand{\Z}{{\mathbb Z}}
\newcommand{\R}{{\mathbb R}}
\newcommand{\RP}{{\mathbb R}P}

\newcommand{\cT}{{\mathcal T}}
\newcommand{\cbT}{{\overline{\mathcal T}}}

\newcommand{\isom}{\mbox{Isom}}
\newcommand{\be}{\begin{eqnarray}}
\newcommand{\ee}{\end{eqnarray}}

\newcommand{\cC}{{\mathcal C}}

\newcommand{\cL}{{\mathcal L}}

\newcommand{\St}{\tilde{S}}

\newcommand{\EM}{\ensuremath}
\newcommand{\norm}[1]{\EM{\left\| #1 \right\|}}

\begin{document}

\title{Flippable tilings of constant curvature surfaces}
\author{François Fillastre}
\address{Universit\'e de Cergy-Pontoise,  d\'epartement de math\'ematiques, UMR CNRS 8088,
F-95000 Cergy-Pontoise, France}
\email{francois.fillastre@u-cergy.fr}
\author{Jean-Marc Schlenker}
\thanks{J.-M. S. was partially supported by the A.N.R. through project
ETTT, ANR-09-BLAN-0116-01.
}
\address{Institut de Math\'ematiques de Toulouse, UMR CNRS 5219 \\
Universit\'e Toulouse III \\
31062 Toulouse cedex 9, France}
\email{schlenker@math.univ-toulouse.fr}
\date{\today (final version)}

\begin{abstract}
We call ``flippable tilings'' of a constant curvature surface a
tiling by  ``black'' and ``white'' faces, so that each edge is
adjacent to two black and two white faces (one of each on each side), 
the black face is 
forward on the right side and backward on the left side, and
it is possible to ``flip'' the tiling by pushing all black
faces forward on the left side and backward on the right side.
Among those tilings we distinguish the ``symmetric'' ones, for
which the metric on the surface does not change under the 
flip. We provide some existence statements, and explain how to parameterize the space of those tilings
(with a fixed number of black faces) in different ways. For
instance one can
glue the white faces only, and obtain a metric with cone
singularities which, in the hyperbolic and spherical case, 
uniquely determines a symmetric tiling. 
The proofs are based on the geometry of polyhedral surfaces
in 3-dimensional spaces modeled either on the sphere
or on the anti-de Sitter space.
\end{abstract}

\maketitle

\section{Introduction}

We are interested here in tilings of a surface which have a striking property:
there is a simple specified way to re-arrange the tiles so that a new tiling 
of a non-singular surface appears. So the objects under study are actually
pairs of tilings of a surface, where one tiling is obtained from the other by
a simple operation (called a ``flip'' here) and conversely. The definition 
is given below.

We have several motivations for studying those flippable tilings. 
One is their intrinsic geometric and dynamic properties. 
It is not completely clear, at first sight,
that those flippable tilings can be ``flipped'', or even that there should
be many examples, beyond a few simple ones that can be constructed by hand.
However the 3-dimensional arguments developed below show that there are indeed
many examples and that the space of all those tilings can be quite well 
understood. Another motivation is that those flippable tilings provide a way
to translate 3-dimensional questions (on polyhedral surfaces) into 2-dimensional
statements.

One of our motivation when developing flippable tilings was to investigate a
polyhedral version of earthquakes. Earthquakes
are defined on hyperbolic surfaces endowed with measured laminations. Those
earthquakes are then intimately related to pleated surfaces in the AdS space (see
\cite{mess,mess-notes} and \cite{earthquakes,cone,mbh} for recent developments).

There are also analogs of this picture for some special smooth surfaces 
in $AdS_3$, in particular maximal surfaces. Earthquakes are then
replaced by minimal Lagrangian diffeomorphisms, see \cite{AAW,minsurf,maximal}.
Similar arguments can be used --- leading to slightly different results --- 
for constant mean curvature or constant Gauss curvature surfaces in 
$AdS_3$ or in 3-dimensional AdS manifolds.

It is then tempting to investigate what happens for polyhedral surfaces,
rather than pleated or smooth surfaces. This leads naturally to the notion
of flippable tilings as developed here, although the analogy with 
earthquakes is only limited, as can be seen in Theorem \ref{tm:earthquake}.

\subsection{Definitions, basic property}

\subsubsection{Definitions and first examples}

\begin{defi}
Let $(S,g)$ be a closed (oriented) surface with a constant curvature metric. 
A {\em right flippable tiling} of $(S,g)$ is a triple $T=(F_b, F_w, E)$, where
\begin{enumerate}
\item $F_b$ and $F_w$ are two sets of closed convex polygons homeomorphic to disks in $S$,
called the ``black'' and ``white'' faces of $T$, with disjoint interiors, which
cover $S$.
\item $E$ is a set of
segments in $(S,g)$, called the {\em edges} of the tiling, such that each
edge of any black or white polygon is contained in some $e\in E$.
\end{enumerate}
We demand that for each oriented edge $e\in E$:
\begin{itemize}
\item $e$ is the union of the edges (in the usual sense) of 
a black and a white face on its right side, and similarly on its left side,
\item the black polygon is forward on the right-hand side of $e$,
and backward on the left-hand side,
\item the lengths of the intersections with $e$ of the two black
faces (resp. the two white faces) are equal.
\end{itemize}
The vertices of the black and white faces are 
the {\em vertices} of the tiling, while the black (resp. white) 
polygons are the {\it black} (resp. {\it white}) {\em faces} of the tiling.

A {\em left} flippable tiling is defined in the same way, but, for
each edge $e$, the black polygon is forward on the left-hand side of
$e$ and backward on the right-hand side. See Figure~\ref{fig.bathroom}.
\end{defi}

\begin{figure}[!h]
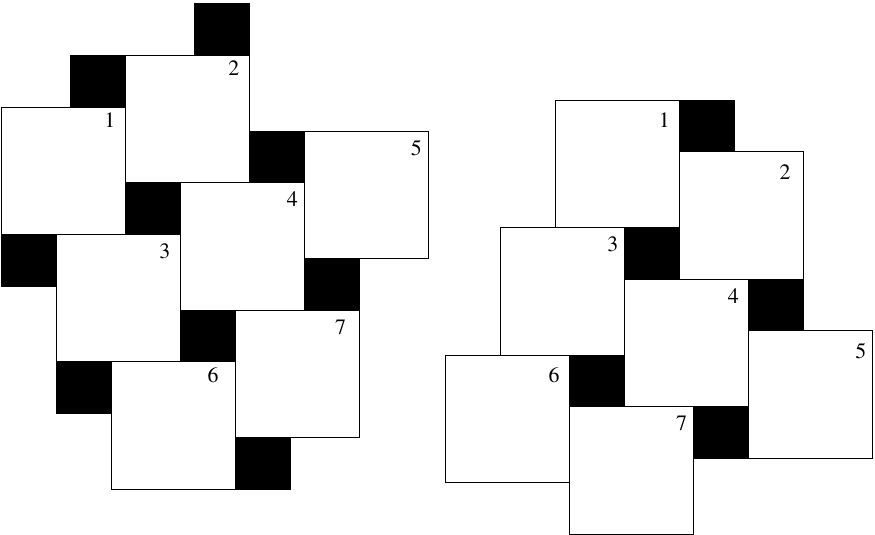
\caption{A flippable tiling, sometimes found in bathrooms --- 
left and right versions.}
\label{fig.bathroom}
\end{figure}

Note that the definition does not depend on the orientation of the 
edges: changing the orientation exchanges both the left and the right
side, and the forward and the backward direction.

We assume that some faces of a flippable tiling on the sphere can be digons 
(see the definitions in Subsection~\ref{sub: def pol sph}).

\begin{example}\label{ex: til deg pol}

Let us consider two distinct geodesics on the round sphere $S^2$.
They divide the sphere into $4$ faces. Paint in black two non-adjacent faces,
and paint in white the other two. We get a flippable tiling
with two edges, and there are two choices for the edges: one leading to a left flippable tiling
and the other to a right flippable tiling.

More generally, consider on $S^2$ a convex polygon $P$ with $n$ edges as well as its antipodal $-P$.
Let $v_1, \cdots, v_n$ be the vertices of $P$, so that $-v_1,\cdots, -v_n$ are the vertices of $-P$.
For each $i\in \{ 1,\cdots, n\}$, let $e_i$ be the edge of $P$ joining $v_i$ to 
$v_{i+1}$, and let $-e_i$ be the edge of $-P$ joining $-v_i$ to $-v_{i+1}$.
For each $i$, draw the half-geodesic starting from $v_i$ along $e_i$, and ending at $-v_i$. 
Then it is not hard to see that we get a decomposition of the sphere into $P$, $-P$ and
digons. The vertices of each digon are a vertex $v_i$ of $P$ and its antipodal $-v_i$. 
One edge of the digon contains an edge of $P$ and the other contains an edge of $-P$.
If we paint in black $P$ and $P'$, and in white all the digons, then 
we get a flippable tiling with $n$ edges.
There is another possibility: drawing for each $i$ the half-geodesic starting from $v_i$
along $e_{i-1}$. One possibility leads to a left flippable tiling
and the other to a right flippable tiling (depending on the orientation of $P$).
\end{example}

Corollary~\ref{lem: goutte} will state that the examples above are the ``simplest'' on the sphere.  
But there exist flippable tilings with only one black and/or white face on higher genus surfaces,
as the following example shows.

\begin{example}
Consider a square. 
Its diagonals divide it in $4$ triangles. 
Paint in black two non-adjacent triangles. 
For a suitable choice of the edges, 
we obtain a flippable tiling of the torus with two faces, 
two edges (which are made of parts of the diagonals) and two vertices. 
This construction also provides examples of flippable tiling for surfaces of higher genus.
\end{example}

\subsubsection{Flipping a flippable tiling}

The terminology used for those tilings is explained by the following basic definition and
statements.

\begin{defi}\label{def:flip}
Let $(S,g_r)$ be a constant curvature closed surface with a right flippable tiling $T_r$.
If there is a metric $g_l$ on $S$ of the same curvature as $g_r$
with a left flippable tiling $T_l$, such that:
\begin{itemize}
\item there is a one-to-one orientation-preserving map $\phi$ between the
faces of $T_r$ and the faces of $T_l$, 
sending each black (resp. white) face $f$ of $T_r$ to a black (resp. white) 
face of $T_l$ isometric to $f$,
\item there is a one-to-one map $\psi$ between the edges of $T_r$ and 
the edges of $T_l$, 
\item if an edge $e$ of a face $f$ of $T_r$ is adjacent to an edge $E$
of $T_r$, then the corresponding edge in $\phi(f)$ is adjacent to
$\psi(E)$, and conversely, 
\end{itemize}
we say that $T_l$ is obtained by {\em flipping} $T_r$. 
Similarly we define the flipping of a left flippable
tiling.
 \end{defi}

In other terms, the left flippable tiling $T_l$ on $(S,g_l)$ has the 
same combinatorics as the right flippable tiling $T_r$ and its faces are isometric to the corresponding
(black or white) faces of $T_l$, but each edge has been ``flipped'' in
the sense that on the right side of each edge, the black face, which
was forward on $T_r$, is backward on $T_l$, see Figure~\ref{fig.flip}. 
Note that the hypothesis that $\phi$ is orientation-preserving is relevant since
otherwise simply changing the orientation on $S$ would yield a 
left flippable tiling with the desired properties.

\begin{figure}[!h]
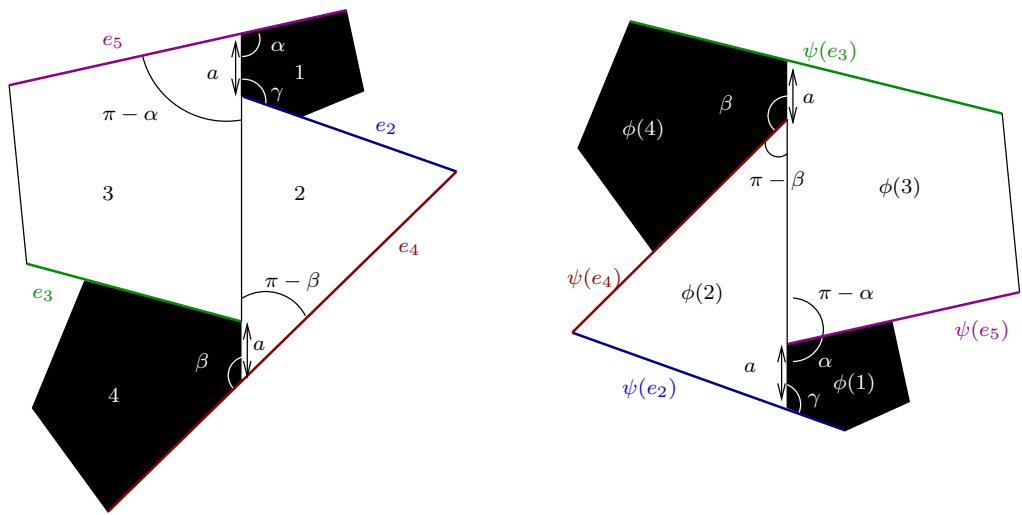
\caption{The ``flip'' of an edge.}
\label{fig.flip}
\end{figure}

\begin{Sthm} \label{tmS:base}
Let $S^2$ be the round sphere  with a right flippable tiling $T_r$.
There is a unique  left flippable tiling $T_l$ on $S^2$, such that $T_l$ is obtained by flipping $T_r$.

Similarly, there is a right flippable tiling obtained by flipping any left flippable
tiling, and flipping a  tiling $T$ twice yields $T$.
\end{Sthm}

\begin{Hthm} \label{tmH:base}
Let $(S,g_r)$ be a hyperbolic surface with a right flippable tiling $T_r$.
There is a unique hyperbolic metric $g_l$ on $S$
with a left flippable tiling $T_l$, such that $T_l$ is obtained by flipping $T_r$.

Similarly, there is a right flippable tiling obtained by flipping any left flippable
tiling, and flipping a  tiling $T$ twice yields $T$.
\end{Hthm}

Those statements are not completely obvious since one has to check that,
after flipping, the faces ``match'' at the vertices. The proof is in
Section~\ref{sc:sphere} for the spherical case, in Section~\ref{sc:ads}
for the hyperbolic case. In both cases it is based on 
\begin{itemize}
\item a notion of duality between equivariant polyhedral surfaces in a 3-dimensional
space, either the 3-dimensional sphere or the 3-dimensional anti-de Sitter
space,
\item some key geometric properties shared by $S^3$ and $AdS_3$, basically 
coming from the fact that they are isometric to Lie groups with a bi-invariant
metric.
\end{itemize}

There is an easy way to get a right (resp. left) flippable tiling from a 
left (resp. right) flippable tiling $T$. It suffices to color in white all
black faces of $T$ and to color in black all white faces of $T$. The flipping obtained
in this way is denoted by $T^*$. Obviously $(T^*)^*=T$.

\subsubsection{Flipping from a hyperbolic metric to another}

Flippable tilings can be considered as ways to deform from one hyperbolic metric --- 
the metric underlying the left tiling --- to another one --- the metric underlying
the right tiling obtained after tiling. Recall that $\cT$ is the Teichm\"uller space
of $S$.

\begin{thm} \label{tm:earthquake}
Let $h_l,h_r\in \cT$ be two hyperbolic metrics on $S$. There is a left flippable tiling $T_l$ 
on $h_l$ such that flipping $T_l$ yields a right flippable tiling $T_r$ on $h_r$.
\end{thm}

The proof can be found in Section \ref{sssc:earthquake}. It should be clear from the proof
that $T_l$ is not unique, on the opposite there are many possible choices, and
finding a ``simplest'' such tiling remains an open question.

\subsection{Moduli spaces of flippable tilings}
\label{ssc:moduli}

We now consider more globally the space of flippable tilings, first on
the sphere and then on closed hyperbolic surfaces of a fixed genus $g\geq 2$.

Let $n\geq 3$, consider a right flippable tiling $T$ with $n$ black faces
on $S^2$. We can associate to $T$ a graph $\Gamma$ embedded in $S^2$ (defined up to
isotopy) as follows: 
\begin{itemize}
\item the vertices of $\Gamma$ correspond to the black faces of $T$,
\item the faces of $\Gamma$ (connected components of the complement of
$\Gamma$ in $S^2$) correspond to the white faces of $T$,
\item the edges of $\Gamma$ correspond to the edges of $T$ (which
bound a black and a white face of $T$ on each side), so that a vertex
of $\Gamma$ is adjacent to a face of $\Gamma$ if and only if the corresponding
black and white faces of $T$ are adjacent to the same edge of $T$.
\end{itemize}
We will call $\Gamma$ the {\it incidence graph} of the flippable tiling $T$.

\begin{defi}
Let $n\geq 3$. We call $\cT^{r,1}_n$ the set of right flippable tilings of
$S^2$ with $n$ black faces, considered up to isotopy. If $\Gamma$ is an embedded graph in $S^2$ 
(considered up to isotopy) with $n$ vertices, we call $\cT^{r,1}_n(\Gamma)$
the set of right flippable tilings of $S^2$ with incidence graph $\Gamma$.
\end{defi}

For each graph $\Gamma$ there is a natural topology on $\cT^{r,1}_n(\Gamma)$,
where two tilings in $\cT^{r,1}_n(\Gamma)$ are close if, after applying an 
isometry of $S^2$, their corresponding vertices are close. However there is 
also a topology on $\cT^{r,1}_n$. We consider that given 
a flippable tiling $t\in \cT^{r,1}_n(\Gamma)$ and a sequence of tilings 
$t_n\in \cT^{r,1}_n(\Gamma')$, then $(t_n)$ converges to $t$ if, after
applying a sequence of isometries of $S^2$, the union of white (resp. black) 
faces of $t_n$ converges to the union of white (resp. black) faces of $t$
in the Hausdorff topology. 

For instance $\Gamma$ could be obtained from
$\Gamma'$ by erasing an edge $e$ between two faces $f,f'$ of $\Gamma'$.
Then $f$ and $f'$ correspond to white faces of $t_n$, for all $n\in \N$. 
In the limit as $n\rightarrow 0$, those faces merge to one white face of
$t$, while the two black faces of the $t_n$ adjacent to the same edge
of the flippable tiling ``lose'' an edge. In a dual way, $\Gamma$ could
be obtained from $\Gamma'$ by collapsing an edge $e$, so that its two 
vertices $v,v'$ become one. In this case the two black faces corresponding
to $v$ and $v'$ merge in the limit flippable tiling. 

\begin{Sthm} \label{tm:moduli-s}
With the notations above:
\begin{enumerate}
\item  When $\Gamma$ is 3-connected, $\cT^{r,1}_n(\Gamma)$
is non-empty and homeomorphic to an open ball of dimension equal to the
number of edges of $\Gamma$.
\item $\cT^{r,1}_n$, with the topology described above, is homeomorphic to a 
manifold of dimension $3n-6$.
\end{enumerate}
\end{Sthm}

Recall that a graph is 3-connected if it cannot be disconnected by 
removing at most 2 vertices --- for instance, the 1-skeleton of a 
triangulation of a simply connected surface is 3-connected. 

We now turn to hyperbolic surfaces. We consider a surface $S$ of genus
$g\geq 2$, and use the notion of incidence graph introduced above for
a flippable tiling. 

\begin{defi}
Let $n\geq 1$. We call $\cT^{r,-1}_{S,n}$ the set of right flippable tilings of
a hyperbolic metric on 
$S$ with $n$ black faces, considered up to isotopy. 
If $\Gamma$ is an embedded graph in $S$ 
(considered up to isotopy) with $n$ vertices, we call $\cT^{r,-1}_{S,n}(\Gamma)$
the set of right flippable tilings of a hyperbolic metric on 
$S$ with incidence graph $\Gamma$.
\end{defi}

As in the spherical case, given a graph $\Gamma$ embedded in $S$, 
there is a natural topology on $\cT^{r,-1}_{S,n}(\Gamma)$, for which 
two hyperbolic $h,h'$ on $S$ and two flippable tilings $T$ and $T'$
respectively in $(S,h)$ and in $(S,h')$, $T$ and $T'$ are close if
there is a diffeomorphism $u:(S,h)\rightarrow (S,h')$ with bilipschitz
constant close to $1$ and which sends each face of $T$ close to the
corresponding face of $T'$ (in the Hausdorff topology). 

However, still as in the spherical case, there is also a natural topology
on $\cT^{r,-1}_{S,n}$, where two flippable tilings $T$ and $T'$ respectively
on $(S,h)$ and on $(S,h')$ are close if there is a diffeomorphism 
$u:(S,h)\rightarrow (S,h')$ with bilipschitz constant close to $1$ sending
the union of the white (resp. black) faces of $T$ close to the union of 
the white (resp. black) faces of $T'$ in the Hausdorff topology. The same
phenomena of merging of two white (resp. black) faces can happen as in the
spherical case, as described above.

\begin{Hthm} \label{tm:moduli-h}
With the notations above:
\begin{enumerate}
\item $\cT^{r,-1}_{S,n}(\Gamma)$ is non-empty as soon as the universal cover of 
$\Gamma$ is 3-connected.
\item Under this condition, $\cT^{r,-1}_{S,n}(\Gamma)$ is
homeomorphic to the interior of a manifold with boundary of dimension $6g-6+e$,
where $g$ is the genus of $S$ and $e$ the number of edges of $\Gamma$.
\item With the topology described above, $\cT^{r,-1}_{S,n}$ is homeomorphic to a 
manifold of dimension $12g+3n-12$.
\item The mapping-class group $MCG(S)$ of $S$ has a canonical 
properly discontinuous action 
on $\cT^{r,-1}_{S,n}$ compatible with the decomposition as the union
of the $\cT^{r,-1}_{S,n}(\Gamma)$. 
\end{enumerate}
\end{Hthm}

Note that it appears likely that $\cT^{r,-1}_{S,n}(\Gamma)$ is always homeomorphic
to a ball. This would follow from an extension to surfaces of higher genus and
to equivariant polyhedra of
a theorem of Steinitz on the topology of the space of convex polyhedra with given combinatorics in Euclidean
3-space.

\subsection{Parameterization with fixed areas of the faces}

We now describe more specific results, on the parameterizations of the 
space of flippable tilings for which the area of the black faces are fixed.

\subsubsection{On the sphere}

\begin{defi}\label{defS:K}
For $n> 1$, let 
$k_1, \cdots, k_n\in (0,\pi)$ have sum less than $4\pi$, and be such that for all
$j\in \{ 1,\cdots, n\}$, $\sum_{i\not= j}k_i>k_j$.
 We denote
by $\cT^{r,1}_n(K)$,  $K=(k_1, \cdots, k_n)$, the space of right flippable tilings on $S^2$ with 
$n$ black faces of area $k_1,\cdots, k_n$.
\end{defi}

\begin{Sthm} \label{tm:Sarea}
Under the conditions of the previous definition, if $n\not= 2$,
there is a natural parameterization of $\cT^{r,1}_n(K)$ by the space of configurations
of $n$ distinct points on $S^2$.
\end{Sthm}

\subsubsection{Hyperbolic surfaces}

\begin{defi}
Let $n\geq 1$, let 
$k_1, \cdots, k_n\in \R_{>0}$ of sum less than $2\pi|\chi(S)|$, let $(S,h)$
be a hyperbolic surface, and let $K=(k_1, \cdots, k_n)$. We denote
by $\cT^{r,-1}_{S,h,n}(K)$ the space of right flippable tilings on $(S,h)$ with 
$n$ black faces of area $k_1,\cdots, k_n$.
\end{defi}

To give a satisfactory hyperbolic analog of Theorem~S\ref{tm:Sarea} we need
to consider a special class of flippable tilings. This notion is
significant only in the hyperbolic or the Euclidean settings. 

\begin{defi} \label{def sym til}
A right flippable tiling $T_r$ of a surface $(S,g_r)$ is {\em symmetric} if the 
metric $g_l$ underlying the left tiling obtained by flipping $T_r$ is isotopic
to $g_r$. We call $\cbT^{r,-1}_{S,h,n}$ the space of symmetric right 
flippable tilings of a hyperbolic surface $(S,h)$ with $n$ black faces, and 
its intersection with $\cT^{r,-1}_{S,h,n}(K)$ is denoted by $\cbT^{r,-1}_{S,h,n}(K)$ .
\end{defi}

Clearly a spherical flippable tiling is always symmetric, since there is
only one spherical metric on $S^2$, up to isometry.

\begin{Hthm}\label{thm: paramADS}
In the setting of the previous definition,
there is a natural parameterization of $\cbT^{r,-1}_{S,h,n}(K)$ by the space of configurations
of $n$ distinct points on $S$.
\end{Hthm}

The proof is based on an analog in the AdS space of a classical Alexandrov
Theorem on polyhedra with given curvature at its vertices, see Section~\ref{sc:alexandrov}.

\subsection{The black and white metrics of a flippable tiling}

\subsubsection{Definitions.}\label{subsub:white}

Consider a flippable tiling of a constant curvature surface $(S,g)$. Let 
$F_b,F_b'$ be two adjacent black faces. Then $F_b$ and $F_b'$ are
adjacent to an edge $e$ which is also adjacent to two white faces $F_w$ 
and $F_w'$. The definition of a flippable tiling shows that $F_w\cap e$ and 
$F_w'\cap e$ have the same length. So there is a well-defined way to glue
$F_w$ to $F_w'$ along their intersections with $e$, isometrically, preserving
the orientation of $e$. Idem for $F_b$ and $F_{b}'$.

\begin{defi}
Let $(S,g_l)$ be a closed surface with a constant curvature metric, and let 
$T_l$ be a left flippable tiling on $(S,g_l)$. The {\em black metric} $h_b$
of $T_l$ is a constant curvature metric with cone singularities obtained by
gluing isometrically the black faces of $T_l$ along their common edges.  
The {\em white metric} $h_w$ of $T_l$ is the metric with cone singularities
obtained by similarly gluing the white faces of $T_l$.
\end{defi}

\begin{remark}
If $(S,g)$ is a hyperbolic surface, then $h_b$ is a hyperbolic metric with 
cone singularities of negative singular curvature (angle larger than $2\pi$).
Each cone singularity $v$ of $h_b$ is associated to a white face $f$ of $T_l$, 
and the angle excess at $v$ is equal to the area of $f$. 

If $(S,g)$ is a flat surface, then $h_b$ is a flat surface (with no cone
singularity). 

If $(S,g)$ is a spherical surface, then $h_b$ is a spherical surface with
cone singularities of positive singular curvature (angle less than $2\pi$). The 
singular curvature at a cone singularity is equal to the area of the 
corresponding white face of $T_l$.
\end{remark}

\subsubsection{Spherical surfaces.}

In the sphere, flippable tilings are in one-to-one correspondence with 
spherical metrics with cone singularities of positive singular curvature. 

\begin{Sthm} \label{tm:sphere}
Let $g$ be a spherical metric on $S^2$ with $n>2$ cone singularities of angle 
less than $2\pi$. There exists a unique left flippable tiling of $S^2$
which has $g$ as its black metric.
\end{Sthm}

\subsubsection{Hyperbolic surfaces.}

We now consider a closed surface $S$ of negative Euler characteristic.
For those hyperbolic surfaces the analog of Theorem~S\ref{tm:sphere}
involves symmetric tilings.

\begin{Hthm} \label{tm:symmetric}
Let $h_+$ be a hyperbolic metric with cone singularities of negative
singular curvature on a closed surface $S$. There is a unique symmetric
flippable tiling $T_+$ on a hyperbolic surface such that $h_+$ is 
the black metric of $T_+$.
\end{Hthm}

The proof of Theorem~H\ref{tm:symmetric}
can be found in Section~\ref{sec:proof sym}.

\subsection{Flippable tilings and polyhedral surfaces}

The results on flippable tilings presented here are consequences, sometimes
very direct, of a correspondence between flippable tilings on
the sphere (resp. a hyperbolic surface) and convex, polyhedral surfaces
in the 3-dimensional sphere (resp. convex, space-like polyhedral surfaces in the
3-dimensional anti-de Sitter space). 

To each convex polyhedron $P$ in $S^3$, we can associate a left flippable tiling
$T$ on $S^2$. The black faces of $T$ correspond to the faces of $P$, while
the white faces of $T$ correspond to the vertices of $P$, so that the incidence
graph of $T$ is the 1-skeleton of $P$. More geometrically, the black faces
of $T$ are isometric to the faces of $P$, while the white faces of $T$
are isometric to the spherical duals of the links of the vertices of $P$.

The same relation holds between flippable tilings on a hyperbolic surface
on one hand, and equivariant polyhedral surfaces in $AdS_3$ on the other.
Symmetric flippable tilings then correspond to Fuchsian polyhedral surfaces. 
The results on the black metrics on tilings correspond to statements on the
induced metrics on polyhedra, going back to Alexandrov \cite{Ale05} in
the spherical case and proved more recently by the first author \cite{fil07,Fil09} 
in the hyperbolic/AdS case. The results on tilings with black faces of fixed
area correspond to statements on polyhedra with vertices having fixed
curvature, which are proved in Section \ref{sc:alexandrov}.

\subsection{Remarks about Euclidean surfaces}

The reader has probably noticed that the results described above
concern hyperbolic surfaces and the sphere, but not flat metrics on
the torus. After a first version of the present paper circulated, it
was noticed by Fran\c{c}ois Gu\'eritaud that 
the  analog of  theorems S\ref{tmS:base} and  H\ref{tmH:base}
is true for these metrics. The proof can be done in an intrinsic way, and moreover 
the flip can be described as a continuous path in the Teichm\"uller space of the torus \cite{guer}. 

Maybe the 3-dimensional spherical or AdS
arguments used here are still valid in this case, considering the space of isometries of the Euclidean plane.
We hope that this question will be the subject of
further investigations.

\subsection*{Acknowledgements}
The authors would like to thank Fran\c{c}ois Gu\'eritaud for useful conversations related to the results
presented here.

\section{Spherical polyhedra}
\label{sc:sphere}

\subsection{Basics about the sphere $S^3$}\label{sub2.1}

\subsubsection{Spherical polyhedra}\label{sub: def pol sph}

Let $S^n$ be the unit sphere in $\mathbb{R}^{n+1}$, endowed with the usual scalar product $\langle \cdotp,\cdotp \rangle$. 
\begin{enumerate}
\item A (spherical) \emph{convex polygon} is a compact subset of $S^2$ obtained by a finite intersection of
closed hemispheres, and contained in  an open hemisphere of $S^2$. 
We suppose moreover that a convex polygon has non-empty interior.
\item  A  \emph{digon} is a spherical polygon is the intersection of two closed hemispheres,
having two vertices and two edges of lengths $\pi$, making an angle  $<\pi$.
\item A (spherical) \emph{convex  polyhedron} is a  compact subset of $S^3$ obtained 
as a finite intersection of half-spaces of the sphere, and contained in  an open hemisphere of $S^3$. We suppose moreover that a convex polyhedron has a non-empty interior. Faces of a convex polyhedron are convex polygons.
\item A \emph{dihedron} satisfies the definition of convex polyhedron except that it has empty interior.
It is a convex polygon $p$ in a totally geodesic plane in $S^3$. We consider its boundary to be
made of two copies of $p$ glued isometrically along their boundary, so that this boundary is
homeomorphic to a sphere.
\item A \emph{hosohedron} is the non-empty intersection
of a finite number of half-spaces with only two vertices. The vertices have to be antipodal, and faces of a hosohedron are digons. A hosohedron is contained in the closure of 
an open hemisphere, it is not 
a convex polyhedron  in regard of the definition above. 
We also call digon the degenerated hosohedron made of two copies of the same digon.
\end{enumerate}

Up to a global isometry, we can consider that all the objects defined above are contained in 
the (closure of) the same open hemisphere $S^3_+$. 

 The \emph{polar link} $\mathrm{Lk}(x)$ of a vertex $x$ of a convex polyhedron $P$ is the convex spherical polygon in $T_xS^3$ given by the outward unit normals of the faces of $P$ meeting at $x$. The lengths of the edges of $\mathrm{Lk}(x)$ are the exterior dihedral angles of $P$ at the edges having $x$ as endpoint, and the interior angles of $\mathrm{Lk}(x)$ are $\pi$ minus the interior angles on the faces of $P$ at $x$. 
The definition of polar link extends to dihedra and hosohedra.

\subsubsection{Polar duality}

Let $x\in S_+^3$. We denote by $x^*$ the hyperplane orthogonal to the vector $x\in\mathbb{R}^4$ for $\langle \cdotp,\cdotp \rangle$. We also denote by $x^*$ the totally geodesic surface intersection between this hyperplane and $S^3$. 
Conversely, if $H$ is a totally geodesic surface of $S^3$, we denote by $H^*$ the 
unique point of $S^3_+$ orthogonal as a vector to the hyperplane $H$ 
(up to an isometry, we can avoid considering $H$ such that its orthogonal unit
vectors belong to the boundary of $S^3_+$).

The \emph{polar dual} $P^*$  of a convex polyhedron $P$ is the convex polyhedron 
defined as the intersection of the half-spaces $\{y\vert \langle x,y \rangle >0  \}$ for each vertex $x$ of $P$.
Each face of $P$ is contained in the dual $x^*$ of a vertex $x$ of $P$.
Equivalently,  $P^*$ is a convex polyhedron whose vertices are the $H^*$ for $H$ a totally geodesic surface  containing a face of $P$. It follows that $(P^*)^*=P$ and
that the polar link $\mathrm{Lk}(x)$ of a vertex $x$ of $P$ is isometric to the face of $P^*$ dual to $x$ and vice-versa.

The definitions of the polar link and polar dual applied to  dihedra and hosohedra indicate that they are in duality with
each other (excepted for digons which are dual of digons).

\subsubsection{A projective model}

Up to a global isometry, we can consider $S^3_+$ as the open subset of  $S^3$ made 
of vectors of $\mathbb{R}^4$ with positive first coordinate. 
By the central projection which sends  $e:=(1,0,0,0)$ to the origin,  $S^3_+$ is identified with  an affine model of the projective space $P^3(\mathbb{R})$ (the plane at infinity is $e^*$).
In this affine model, a spherical convex polyhedron is drawn as an Euclidean convex polytope, 
a dihedron is drawn as a convex Euclidean polygon in a plane, a hosohedron is a convex polyhedral infinite cylinder, 
and a digon is the strip between two parallel lines. 

\subsubsection{Multiplicative structure on the sphere}

The following isometry gives a multiplicative structure on $\mathbb{R}^4$:

$$ (\mathbb{R}^4, \langle \cdotp,\cdotp \rangle)\rightarrow (\mathcal{M},\mathrm{det}),
\quad(x_1,x_2,x_3,x_4)\mapsto \begin{pmatrix}
         x_1+ix_2 & x_3+ix_4 \\
         -x_3+ix_4 & x_1-ix_2
\end{pmatrix}.
 $$
where $\mathcal{M}$ is the space of matrix 
$\begin{pmatrix}
         \alpha & -\overline{\beta} \\
         \beta & \overline{\alpha}
\end{pmatrix}.
 $ 
 Note that the neutral element is $e$. 
It follows that $S^3$ is isometric to $SU(2)$, and that multiplication on the right or on the left
by an element of $S^3$ is an isometry. As the kernel 
of the classical homomorphism $SU(2)\rightarrow SO(3)$ is $\{ -\mbox{Id}, +\mbox{Id}\}$, $P^3(\mathbb{R})$ is 
isometric to $SO(3)$, and then $SO(3)$ acts by isometries on $P^3(\mathbb{R})$ by left and right multiplication.

\begin{lemma}\label{lem: basic mult}
 \begin{enumerate}
  \item The left multiplication by $x$ sends $y^*$ to $(xy)^*$.
\item If $x\in e^*$, then $\langle x,y\rangle=-\langle x,y^{-1}\rangle$. In particular 
$\langle x,y \rangle=0$ if and only if $\langle x,y^{-1} \rangle=0$.
 \end{enumerate}
\end{lemma}
\begin{proof}
  \begin{enumerate}
  \item If $z\in y^*$ then $\langle z, y\rangle=0$, and the left multiplication is an isometry, so
$\langle xz, xy \rangle=0 $, i.e.~the image of $z$ by the left multiplication by $x$ belongs to $(xy)^*$.
\item It is easy to check that $x=(x_1,x_2,x_3,x_4)\in e^*$ is equivalent to $x_1=0$ and that if 
$y=(y_1,y_2,y_3,y_4)$ then $y^{-1}=(y_1,-y_2,-y_3,-y_4)$.
 \end{enumerate}
\end{proof}

\subsection{Flippable tilings on the sphere and convex polyhedra}\label{sub2.2}

\subsubsection{The left and right projections of an angle}

Let $a,b$ be in $S^3_+$, with $a,b,e$ and their antipodals mutually distinct.
Let $A$ be the angle formed by $a^*$ and $b^*$, with edge $E$:
\begin{eqnarray*}
 \ A&=&a^*_c\cup b^*_c\cup E \\
\ &=& \{x \in S^3\vert \langle x,a\rangle=0,\langle x,b\rangle > 0 \}\cup\{x \in S^3\vert \langle x,b\rangle=0,\langle x,a\rangle > 0 \}\cup \{x\in S^3\vert \langle x,a \rangle=\langle x,b \rangle=0 \}.
\end{eqnarray*}

The \emph{left projection} of the angle $A$ is the following  map:

$$
L_A :  A\setminus E  \longrightarrow  e^*,
     x   \longmapsto 
\left\{
     \begin{matrix}
   a^{-1}x \mbox{ if } x\in a_c^*    \\
    b^{-1}x \mbox{ if } x\in b_c^*    \\
            \end{matrix}
\right. .
$$

The \textit{right projection}  of the angle $A$ is defined similarly:
$$
R_A :  A\setminus E  \longrightarrow  e^*,
     x   \longmapsto 
\left\{
     \begin{matrix}
   xa^{-1} \mbox{ if } x\in a_c^*    \\
    xb^{-1} \mbox{ if } x\in b_c^*    \\
            \end{matrix}
\right. .
$$

In the following we will identify $e^*$ with $S^2$: $L_A$ and $R_A$ have now values in the sphere
$S^2$.

\begin{lemma}\label{lem: proj angle} 
\begin{enumerate}
\item The half-planes $a^*_c$ and $b^*_c$ are isometrically mapped by $L_A$ to two non-intersecting open hemispheres of $S^2$ delimited by the great circle $E^l:=a^{-1}E=b^{-1}E$. 
If $x\in E$ then the distance between $a^{-1}x$ and  $b^{-1}x$ 
is the exterior dihedral angle of $A$. Moreover, if   $L_A(a^*_c)$ is on the left of $E^l$
 and  $L_A(b^*_c)$ is on the right of $E^l$, then  $a^{-1}x$ is forward and 
 $b^{-1}x$ is backward. 
\item   The half-planes $a^*_c$ and $b^*_c$ are isometrically mapped by $R_A$ to 
two non-intersecting open hemispheres of $S^2$ delimited by the great circle  $E^r:=Ea^{-1}=Eb^{-1}$. 
If $x\in E$ then the distance  between $xa^{-1}$ and  $xb^{-1}$ 
is the exterior dihedral angle of $A$. Moreover, if   $R_A(a^*_c)$ is on the left of $E^r$
 and  $R_A(b^*_c)$ is on the right of $E^r$, then $b^{-1}x$ is forward and 
 $a^{-1}x$ is backward. 
\end{enumerate}
\end{lemma}
\begin{proof}
 We prove the first part of the lemma, the second is proved using 
similar arguments. We first check that $L_A$ is an injective map. It suffices
to show that $L_A(a_c^*)\cap L_A(b^*_c)=\varnothing$. Suppose that the intersection is non-empty. 
Hence there exists $x_1\in a^*_c$ and $x_2\in b^*_c$ such that $a^{-1}x_1=b^{-1}x_2$.
From $\langle x_1,b\rangle >0$ it comes $\langle b^{-1}x_2,a^{-1}b \rangle >0$, but $b^{-1}x_2\in e^*$, so
by Lemma~\ref{lem: basic mult} we obtain $\langle x_2,a\rangle <0$, that is a contradiction.

 Let $x\in E$. Lemma~\ref{lem: basic mult} implies that  $a^{-1}x \in (a^{-1}b)^*\cap e^*$, $b^{-1}x \in (b^{-1}a)^*\cap e^*$,
 and that these two spaces are equal.
  By definition, the cosine of the exterior dihedral angle $\alpha < \pi$ of $A$ is $\cos \alpha = \langle a,b\rangle= \langle a^{-1},b^{-1}\rangle=\langle a^{-1}x,b^{-1}x \rangle=\cos d_{S^2}(a^{-1}x,b^{-1}x)$. 

It is not hard 
to see that $a^{-1}b$ is a vector orthogonal  to $E^l$  on the same side as $L_A(a^*_c)$.
It remains to prove that, for $x\in E$, $\mathrm{det}(b^{-1}x,a^{-1}x,a^{-1}b,e)>0$. This determinant never vanishes. This comes
from the facts that   $a^{-1}b$ and $e$ are orthogonal to
$E^l\subset e^*$,  $a^{-1}x$ and $b^{-1}x$ span $E^l$, and $a^{-1}b$ is not collinear to $e$.
Hence the sign of the determinant is constant, and it can be checked on examples.
\end{proof}

\subsubsection{Convex polyhedra as flippable tilings}

Let $P$ be a convex polyhedron and $E$ an edge of $P$. 
The planes containing the faces  of $P$ adjacent to $E$ make an angle $A(E)$. 
The \emph{left projection}  of $P$ is,
for each edge $E$ of $P$, the left projection of the angle $A(E)$  restricted 
to the faces of $P$ adjacent to $E$. The right projection  of $P$ is defined in a similar way.

The following two propositions are the keystone of your study of flippable tilings in the
spherical case, since they relate flippable tilings to convex polyhedra.

\begin{prop}\label{prop: proj sph}
The left (resp. right) projection of a convex polyhedron $P$ is the set of the white faces of a right 
(resp. left) flippable tiling of the sphere. The black faces are isometric to the polar links of the vertices of $P$.
The 1-skeleton of $P$ is isotopic to
the incident graph of the tiling.
\end{prop}

\begin{proof}
Let us consider the case of the left projection. The case of the right projection is proved in a similar way.
 The  left projection of the faces of $P$ adjacent to an edge $E$
gives two convex polygons $w_1$ and $w_2$ of $S^2$, that we paint in
white. From Lemma~\ref{lem: proj angle}, $w_1$ and $w_2$ have an edge  on a common geodesic. 
Let us consider the part $e$ of this geodesic which contains 
the edge of each  polygon and which is such that one of them is forward on the left.

Let $s$ be a vertex of $P$ and $E_1,\cdots,E_n$ be  (oriented) edges of $P$ having $s$ as an endpoint.
As an element of the edge $E_i$, the left projection sends $s$ to two points $\mathrm{ext}_{i}$ and $\mathrm{int}_{i}$
 of the segment
$e_i$ of $S^2$ ($e_i$ is defined as above).  
 $\mathrm{ext}_{i}$ is an endpoint of $e_i$ and $\mathrm{int}_{i}$ belongs to the interior of $e_i$. 
Similarly, if $E_j$ is another edge from $s$ adjacent to a face having $E_i$ has an edge, $s$ is sent to two points
$\mathrm{ext}_{j}$ and $\mathrm{int}_{j}$ of $e_j$, 
 and $\mathrm{ext}_{j}=\mathrm{int}_{i}$ or $\mathrm{int}_{j}=\mathrm{ext}_{i}$.
Finally  $s$ is sent by the left projection to $n$ distinct points of 
$S^2$. The segments between $\mathrm{ext}_{i}$ and $\mathrm{int}_{i}$ for any $i$ define a polygon $p$. 
Lemma~\ref{lem: proj angle} says that the lengths of the edges 
of $p$ are equal to the dihedral angles of $P$.  
Moreover, by construction, any  interior angle of $p$ is supplementary 
to the interior angle of a white face, which is the angle of a face of $P$ at $s$. By definition of the polar link, $p$ is isometric to the
polar link  of $P$ at $s$. In particular it is a convex polygon. We paint it in black.

If we do the same for all the vertices of $P$, we get a set of white faces and a set of black faces on $S^2$.
By construction all faces have disjoint interior. 
 The Gauss--Bonnet formula says exactly that the area of $P$ plus the area of the polar duals of the vertices of $P$ is
$4\pi$, the area of $S^2$. Hence our white and black faces cover the whole sphere. The definition of right 
flippable tiling is satisfied, and it clearly satisfies the properties stated in the proposition.
\end{proof}

We denote by $T_l(P)$ (resp. $T_r(P)$) the left (resp. right) flippable tiling obtained from $P$. Notice that if $Q$ is 
a convex polyhedron isometric to $P$, then $T_l(Q)$ is isometric to $T_l(P)$. In the following, we may write such equivalence as: if $P=Q$ then
$T_l(P)=T_l(Q)$.
 It comes from the definitions that 
$$T_l(P^*)=T_r(P)^*; \quad T_r(P^*)=T_l(P)^*.$$

\begin{remark}
Proposition~\ref{prop: proj sph}  remains true for 
hosohedra which are not digons. It is actually also true for dihedra and digons --- up to adapting
 Lemma~\ref{lem: proj angle} in the case where $a=-b$. The flippable tilings obtained are the ones described in Example~\ref{ex: til deg pol}.
\end{remark}

\subsubsection{Flippable tilings as convex polyhedra}

Let $T_l$ be a left flippable tiling of $S^2$.
It is always possible to  isometrically embed in $S^3$ the white faces of $T_l$ adjacent to 
a black face $b$, such that the resulting gluing is a convex cone with boundary,
with polar link at the apex isometric to $b$, see Figure~\ref{fig:cone}. Moreover, up to
congruence, the resulting cone is unique.

\begin{figure}[ht]
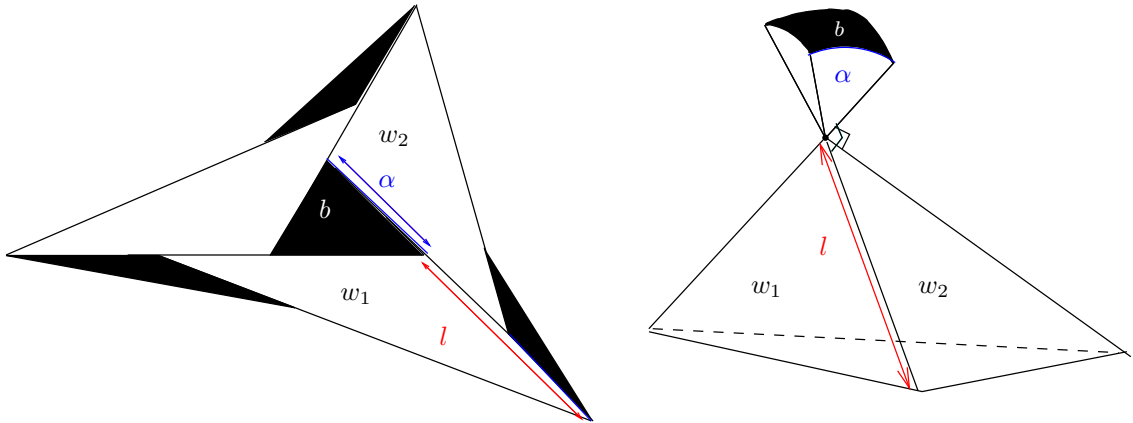
\caption{White faces around a black face define a convex cone.}\label{fig:cone}
\end{figure}

We repeat this procedure for the  black faces connected to $b$, and so on. 
 By the principle of global continuation of functions on simply connected manifolds 
(see for example \cite[Lemma 1.4.1]{Jos06}), this provides a well defined map from $T_l$ to $S^3$.
This map is clearly continuous,  and, as it is clear from the definition that
the incidence graph of a flippable tiling 
is connected, we obtain a closed (i.e.~compact without boundary) polyhedral surface of $S^3$,
 called the \emph{white polyhedron} of $T_l$ and denoted by $P_w(T_l)$.

\begin{prop}\label{pr:white}
The white  polyhedron $P_w(T_l)$ of a left flippable tiling $T_l$ of the sphere $S^2$ is a convex polyhedron
(or a dihedron or a hosohedron) of $S^3$, 
unique up to a global isometry. The 1-skeleton of $P_w(T_l)$ is isotopic to
the incident graph of $T_l$. Moreover a white face of $T_l$ is isometric to its corresponding face 
of $P_w(T_l)$ and a black face of $T_l$ is isometric to the polar link of the corresponding vertex of 
$P_w(T_l)$.
\end{prop}

It will be  clear from the proof that if $T_l$ has exactly two black (resp. white) faces then  $P_w(T_l)$ is a hosohedron
(resp. a dihedron).

\begin{proof}

By construction, $P:=P_w(T_l)$  satisfies the properties stated in the proposition.
It remains to check that this polyhedral surface is convex.

Let us denote by $M$ the white metric of $T_l$ (see~\ref{subsub:white}). Then
$P$ is an isometric polyhedral immersion of $M$ into $S^3$. Let us denote by $i$
this immersion. Let $p$ be a cone singularity of
$M$ and let $V$ be the set of white faces on $M$ having $p$ as vertex. By definition
$i(V)$ is a convex cone in $S^3$. 
Note that as the metric on $T_l$ is the round metric on $S^2$, the distance
between two conical singularities of $M$ is $\leq \pi$. 

Let $H$ be an open hemisphere
which contains $i(V)$ and such that $i(p)$ belongs to the boundary $\partial H$
(i.e.~$\partial H$ is a support plane of $i(V)$). 
Suppose that the distance between $p$ and another conical singularity $p'$
belonging to the closure $\overline{V}$ of $V$ is $\pi$. As the faces of the tiling are supposed to 
be convex spherical polygons or digons, this means that $p$ and $p'$ are two vertices of a digon.
Each edge of the digon is the edge of a white face, and by definition of the tiling the 
white faces next to this one have to be digons. Finally the cone with apex $i(p)$ is a hosohedron
and this is actually the whole $P$.
 
Now we suppose that  $i(\overline{V})$ is contained in $H$. 
We will prove that in this case $P\setminus \{i(p)\}$ is contained in $H$. 
Suppose it is true. Then there exists an affine chart of $S^3$ in which 
$P$ is  a locally convex compact polyhedral surface, and it follows that $P$ is convex 
 \cite{VH52}. Hence  $P$ is a dihedron  or a convex polyhedron, depending if it has
empty interior or not.
Let us   prove that $P\setminus \{i(p)\}$ is contained in $H$.
This
is a straightforward adaptation to the polyhedral case of the argument of \cite{dCW70}, which  concerns
differentiable surfaces. For convenience we repeat it. 

We denote by $\pi$ the map from $H$ to the affine chart of $S^3$ defined by $H$
and by $N$ the connected component of $i^{-1}(H)$ containing $V\setminus \{p\}$.
 $\pi(i(N))$ is a locally convex subset of the Euclidean space, and as $M$ is compact,   $\pi(i(N))$  is
 complete for the metric induced by the Euclidean metric \cite[2.5]{dCW70}. Moreover 
$\pi(i(N))$ contains the image of a neighborhood of a conical singularity $q$ of $M$ different from $p$
(take a conical singularity contained in the boundary of $V$). In the Euclidean space, the
image of $q$ is the apex of a truncated convex cone contained in the image of $N$.  
Finally  $\pi(i(N))$ is unbounded ($\pi$ sends $i(p)$ to infinity). By \cite{VH52}, 
$\pi(i(N))$ is homeomorphic to a plane, then
$N$ is homeomorphic to a plane.
It follows that, as $\partial V$ is homeomorphic to a circle, its image for
$\pi\circ i$ separates $\pi(i(N))$ into two connected components. Let us denote by $W$ the bounded one.
It is easy to see that if $m\in N\setminus V$, then $\pi(i(m))\in W$ (suppose the converse and look at
the image of a segment
between $m$ and $p$). This implies that no point of $\partial H$, except $i(p)$, is a limit point of
$i(N)$. As $M$ is connected it follows that  $P\setminus \{i(p)\}$ is contained in $H$.
\end{proof}

The two constructions above (from a polyhedron to a flippable tiling and from a  flippable tiling to a polyhedron) 
are of course inverse one from the other, for example for left flippable tilings:
\begin{equation}\label{eq:white pol}
T_l(P_w(T_l))=T_l, \quad P_w(T_l(P))=P.
\end{equation}

The \emph{black polyhedron} $P_b(T_l)$ of a left flippable tiling $T_l$ is constructed similarly, 
but the faces of $P_b(T_l)$ are isometric to the black faces of $T_l$, and the polar links at the vertices of  $P_b(T_l)$
are isometric to the white faces of $T_l$. The definitions of white and black polyhedron
for a right flippable tiling are straightforward. We have for example:

$$ P_b(T_l)=P_w(T_l)^*=P_w(T_l^*). $$

\subsection{Proof of Theorem~S\ref{tmS:base}}

Let $T_r$ be a right flippable tiling on the sphere. 
The  left flippable tiling
$$T_r':= T_l(P_w(T_r)) $$
is obtained by flipping $T_r$. Similarly, if $T_l$ is a left flippable tiling on the sphere, 
the  right flippable tiling
$$T_l':= T_r(P_w(T_l)) $$
is obtained by flipping $T_l$. Moreover $(T_r')'=T_r$ due to (\ref{eq:white pol}).

\subsection{Proof of Theorem~S\ref{tm:moduli-s}}

\subsubsection{Flippable tilings vs convex polyhedra}

Proposition~\ref{prop: proj sph} and Proposition~\ref{pr:white} determine a 
one-to-one map between flippable tilings on the sphere $S^2$ and convex polyhedra
in $S^3$. Moreover the incidence graph of a flippable tiling is equal to the
1-skeleton of the corresponding convex polyhedron. 

There is a strong relation between on the one hand the topology on the space
of flippable tilings on the sphere (as defined right before Theorem S\ref{tm:moduli-s})
and on the other hand the natural topology on the space of convex polyhedra in 
$S^3$ (the Hausdorff topology on the interior of the polyhedra, considered 
up to global isometry). 

\begin{lemma} \label{lm:topologies}
Let $(T_{l,n})_{n\in \N}$ be a sequence of left flippable tilings on $S^2$, with the
same number of black faces. This 
sequence converges to a left flippable tiling $T_l$ if and only if the sequence
of the white polyhedra $P_w(T_{l,n})$ converges to a convex polyhedron $P$. $P$ is then the
white polyhedron of the limit $T_l$.
\end{lemma}
In the statement above and in the proof below, $P$ is a spherical convex polyhedron in a wide sense: it can also be
a dihedron or a hosohedron. 
\begin{proof}
Clearly we can suppose that all the polyhedra $P_w(T_{l,n})$ have the same 
combinatorics. However the combinatorics of the limit $P$ might be different.

Suppose that $(P_w(T_{l,n}))$ converges to a limit $P$. For all $n$, the white faces of 
$T_{l,n}$ correspond to the faces of $P_w(T_{l,n})$, and the set of faces of $P_w(T_{l,n})$
converges to the set of faces  
 of $P$. The black faces of $T_{l,n}$ correspond 
to  the polar links of the vertices of $P_w(T_{l,n})$, which are the faces of the dual
polyhedron $P_w(T_{l,n})^*$. However the polar map is continuous \cite[2.8]{RH93} and thus 
$P_w(T_{l,n})^*\rightarrow P^*$, so that the faces of the polyhedra $P_w(T_{l,n})^*$ 
converge to the faces of $P^*$. It follows that the flippable tilings $T_{l,n}$
converge, in the topology defined before Theorem S\ref{tm:moduli-s}, to a
left flippable tiling $T_l$, with $P=P_w(T_l)$. 

The same argument can be used conversely, for a sequence of flippable tilings 
$(T_{l,n})$ converging to a limit $T$. Then both the faces and the dual faces of
the corresponding polyhedra converge, so that the sequence of those polyhedra has
to converge. 
\end{proof}

The same lemma, along with its proof, holds also for flippable tilings on hyperbolic
surfaces, in relation to equivariant polyhedral surfaces in $AdS_3$.

\subsubsection{The proof}

To prove that $\cT^{r,1}_n(\Gamma)$
is non-empty it is sufficient to show the existence of a convex polyhedron
$P\subset S^3$ with 1-skeleton equal to $\Gamma$. However according to a well-known theorem of
Steinitz, any 3-connected planar graph is the 1-skeleton of a convex polyhedron
$Q\subset \R^3$. One can then consider the image of $Q$ in a projective model of 
a hemisphere $H\subset S^3$, which is a convex polyhedron in $H$ with the same
combinatorics as $Q$. This proves that $\cT^{r,1}_n(\Gamma)$ is non-empty as soon
as $\Gamma$ is 3-connected.

Under this assumption, and according to Lemma \ref{lm:topologies},
the topology of $\cT^{r,1}_n(\Gamma)$ is the same as the topology of 
the space of convex polyhedra in $S^3$ with 1-skeleton equal to $\Gamma$, considered
up to global isometries. But the
same argument based on projective models of the hemisphere shows that it is the same
as the topology of the space of convex polyhedra in $\R^3$ with 1-skeleton equal to 
$\Gamma$, because any convex polyhedron in $S^3$ is contained in a hemisphere
and is therefore the image in the projective model of this hemisphere of a 
convex Euclidean polyhedron. Finally it is known that the space of 
convex Euclidean polyhedra of given combinatorics is contractible, this is
a refinement of the Steinitz theorem, see \cite[Section 13.3, p. 144]{richter-gebert}.
So $\cT^{r,1}_n(\Gamma)$ is homeomorphic to a ball. 

Still by Lemma \ref{lm:topologies}, $\cT^{r,1}_n$ is homeomorphic to the space
of convex polyhedra with $n$ vertices in $S^3$, considered up to global 
isometry. This space is itself homeomorphic to a manifold of dimension 
$3n-6$ --- the space of $n$-tuples of points in convex position in the sphere --- 
and point (2) follows.

\subsection{Proof of Theorem~S\ref{tm:sphere}}\label{sub: proofS4}

Theorem~S\ref{tm:sphere} is a direct consequence of the  following famous theorem.

\begin{thm}[{Alexandov Existence and Uniqueness Theorem, \cite{Ale05}}]\label{thm: alexandrov}
 Let $g$ be a spherical metric with $n>2$ conical singularities of positive curvature on the sphere $S^2$. Then there exists a convex polyhedron $P_g$ such that the induced metric on the boundary of $ P_g$ is isometric to $g$. 
Moreover $P_g$ is unique up to global isometries.
\end{thm}

Let us be more precise. If $P$ is a convex polyhedron, then it is clear that the black metric 
of $T_r(P)$ is the induced metric on the boundary of $P^*$, or equivalently that the 
black metric of $T_r(P)^*$ is the induced metric on the boundary of $P$.

Now let $g$ and $P_g$ be as in the statement above. Then $T_r(P_g)^*=:T_l$ is a left flippable tiling 
 of the sphere, and its black metric is the induced metric on $P_g$, that is $g$. This proves the existence part
of Theorem~S\ref{tm:sphere}. Moreover the uniqueness part of Theorem~\ref{thm: alexandrov} implies that $P_g=P_b(T_l)$, that implies the uniqueness part of Theorem~S\ref{tm:sphere}.

\begin{remark}
 Theorem~\ref{thm: alexandrov} is false for metrics with two conical singularities. Actually
for such metrics the cone points have to be antipodal (see Theorem~\ref{thm: tro2} below), and then they are realized in $S^3$ by hosohedra, but it is easy to construct two non-congruent hosohedra having the same induced metric. It follows that for $n=2$ the existence part of Theorem~S\ref{tm:sphere} is true, but the uniqueness part is false. The case $n=1$ is meaningless as  there doest no exist spherical metric on the sphere with
only one conical singularity (Theorem~\ref{thm: tro2}).
\end{remark}

\begin{thm}[{\cite{Tro89,Tro93}}]\label{thm: tro2}
Let $g$ be a spherical metric on the sphere  with two conical singularities of positive curvature. Then
the two singularities are antipodal and the curvature of both singularities are equal.
Furthermore, any two such surfaces are isometric if and only if they have the same
curvature.

A spherical metric on the sphere  cannot have only one conical singularity.
 \end{thm}

\begin{cor}\label{lem: goutte}
The only examples of flippable tiling on the sphere with two black faces are the ones described in Example~\ref{ex: til deg pol}. 
Moreover
there does not exist any flippable tiling on the sphere with only one black face or only one white face.
\end{cor}
\begin{proof}
Consider the white metric of a flippable tiling on the sphere with two black faces. We obtain 
a spherical metric on the sphere with two conical singularities of positive curvature.
Theorem~\ref{thm: tro2} says that the singularities have to be antipodal, hence the 
black faces of the tiling have to be antipodal. 
\end{proof}

\subsection{Proof of Theorem~S\ref{tm:Sarea}}

Let $\mathcal{M}^1_n(K)$ be the space of spherical metrics on the sphere with $n$ conical
singularities of positive curvature $k_i$ ($K$ is as in Definition~\ref{defS:K}).

There is a map from $\cT^{r,1}_n(K)$ to $\mathcal{M}^1_n(K)$ which consists of taking the 
white metric. This map is bijective as its inverse is given by (a suitable version) of
Theorem~S\ref{tm:sphere}. The proof of Theorem~S\ref{tm:Sarea} follows because by the following theorem 
$\mathcal{M}^1_n(K)$ is parameterized by the space of configurations
of $n$ distinct points on $S^2$.
 
\begin{thm}[{\cite{LT92}}]
 Let $k=(k_1,\cdots, k_n)\in (\R_+)^n$, $n>2$, such that $\sum_{i=1}^n k_i < 4\pi$ and $\sum_{i\not= j}k_i>k_j$. For each conformal structure on the $n$-punctured sphere 
 there 
exists a unique conformal spherical metric  with cone points of curvature $k_i$.
\end{thm}
(The existence part was obtained in \cite{Tro91}.) Note that Theorem~S\ref{tm:Sarea} is meaningless for $n=1$ and  false for $n=2$ due to Theorem~\ref{thm: tro2}.

\section{Some AdS geometry}
\label{sc:ads}

In this section we recall the definition and some basic properties of the 
anti-de Sitter space $AdS_n$, and in particular of $AdS_3$. We point out in
particular those properties reminiscent of those of the 3-dimensional sphere,
which are most relevant for flippable tilings.

\subsection{The $n$-dimensional anti-de Sitter space and its projective model}

The anti-de Sitter (AdS) space can be considered as the Lorentzian analog
of the hyperbolic space. It can be defined as a quadric in $\R^{n-1,2}$,
which is simply $\R^{n+1}$ with a symmetric bilinear form of signature
$(n-1,2)$:
$$ \langle x,y\rangle_2=x_1y_1+\cdots+ x_{n-1}y_{n-1}-x_{n}y_{n}-x_{n+1}y_{n+1}~, $$
then
$$ AdS_n := \{ x\in \R^{n-1,2}~|~ \langle x,x\rangle_2 = -1\}~, $$
with the induced metric.
Here are some of its key properties.
\begin{itemize}
\item $AdS_n$ is a Lorentzian space of constant curvature $-1$,
\item it is geodesically complete,
\item it is {\it not} simply connected, its fundamental group is $\Z$ and
its universal cover is denoted by $\tilde{AdS_n}$,
\item its totally geodesic space-like hyperplanes are all isometric to the hyperbolic space
$H^{n-1}$,
\item its time-like geodesics are closed curves of length $2\pi$,
\item there is a well-defined notion of ``distance'' between two points,
defined implicitly by the relation:
$$\cosh d(x,y)=-\langle x,y \rangle_2~. $$
This distance is real and positive when $x$ and $y$ are connected by
a space-like geodesic, imaginary when $x$ and $y$ are connected by a
time-like geodesic. It is in both cases equal to the length of the
geodesic segment connecting them.
\end{itemize}

It is sometimes also convenient to consider the quotient of $AdS_n$ by
the antipodal map (sending $x$ to $-x$ in the quadric model above), we
will simply denote it by $AdS_n/\Z_2$ here. In this space, time-like
geodesics are closed curves of length $\pi$. 

As for $H^n$, there is a useful projective model of $AdS_n$, obtained
by projecting in $\R^{n-1,2}$ from the quadric to any tangent plane in the direction
of the origin. However this yields a projective model of only ``half''
of $AdS_n$ (as for the $n$-dimensional sphere). For instance, 
the projection on the hyperplane $H$ of equation $x_{n+1}=1$ is a 
homeomorphism from the subset of $AdS_n$ of points where $x_{n+1}>0$ 
to its image, which is the interior of the intersection with $H$ of the cone of
equation $\langle x,x\rangle_2=0$ in $\R^{n-1,2}$. This intersection
is again a quadric, but now in $\R^n$. 

This model can be considered in the projective $n$-space 
$\R P^n$, it provides a projective model of the whole quotient 
$AdS_n/\Z_2$ as the interior of a quadric of signature $(n-1,1)$. 
Lifting to the double cover of $\R P^n$ we find a projective
model of the whole space $AdS_n$ as the interior of a quadric 
in $S^n$.

This projective model naturally endows $AdS_n$ with a ``boundary
at infinity'', topologically $S^{n-2}\times S^1$. This boundary at 
infinity comes with a conformal Lorentz structure, namely the conformal
class of the second fundamental form of the quadric in $\R^n$ (or
equivalently $\R P^n$, resp. $S^n$). 

\subsection{The 3-dimensional space $AdS_3$}

As a special case of the $n$-dimensional case we find the 3-dimensional
AdS space which will be of interest to us in the sequel:
$$AdS_3:=\{x\in\mathbb{R}^{2,2} \vert \langle x,x\rangle_2 =-1\}.$$
It has some specific features which are not present in the higher-dimensional
case. In some respects $AdS_3$ can be considered as a Lorentz analog of the
$3$-dimensional sphere.
\begin{itemize}
\item $AdS_3$ is isometric to $SL(2, \R)$ with its bi-invariant Killing metric.
(Similarly, $S^3$ is identified with $SU(2)$).
\item Therefore, it has an isometric action of $SL(2,\R)\times SL(2,\R)$, where
the left factor acts by left multiplication and the right factor acts by right
multiplication. The kernel of this action is reduced to $\{ (I,I),(-I,-I)\}$ 
and this identifies the identity component of the 
isometry group of $AdS_3$ with $SL(2,\R) \times SL(2\R)/\Z_2$.
(There is a similar isometric action of $O(3)\times O(3)$ on $S^3$.)
\item There is a notion of duality in $AdS_3$ reminiscent of the polar duality
in $S^3$. It associates to each oriented totally geodesic plane in $AdS_3$ a point, 
and conversely. The dual of a convex polyhedron is a convex polyhedron, with the 
edge lengths of one equal to the exterior dihedral angles of the other. 
\end{itemize}

The decomposition of the isometry group of $AdS_3$ as $SL(2,\R)\times SL(2,\R)/\Z_2$
can also be understood as follows. In the projective model above, $AdS_3/\Z_2$ is identified
with the interior of a quadric $Q\subset \RP^3$. $Q$ is foliated by two families of
projective lines, which which we call $\cL_l$ and $\cL_r$ and call the ``left'' and
``right'' families of lines. Each line of each family intersects each line of the
other family at exactly one point, and this provides both $\cL_l$ and $\cL_r$ with
a projective structure. The isometries in the identity component of 
$\isom(AdS_3/\Z_2)$ permute the lines of $\cL_l$ and of $\cL_r$, and the corresponding
action on $\cL_l$ and $\cL_r$ are projective, so they define two elements of 
$PSL(2,\R)$.

\subsection{Globally hyperbolic AdS manifolds}\label{ssc:gh}

We recall in this section some known facts on globally hyperbolic AdS 3-manifolds,
mostly from \cite{mess}, which are useful below.

Let $S$ be a closed surface of genus at least $2$. We denote by $\isom_0(AdS_3)$ the
identity component of the isometry group of $AdS_3$.

An AdS 3-manifold is a Lorentz manifold with a metric locally modelled on the 
3-dimensional AdS space, $AdS_3$. It is {\it globally hyperbolic} (GH) if 
it contains a closed space-like surface $S$, and any inextensible time-like
curve intersects $S$ at exactly one point. It is {\it globally hyperbolic maximal}
(GHM) if it is maximal --- for the inclusion --- under those conditions. 

Globally hyperbolic AdS manifolds are never geodesically complete. However, a 
GHM AdS manifold is always the quotient by $\pi_1(S)$ of a convex subset $\Omega$ of 
$AdS_3$. This defines a representation $\rho$ of $\pi_1(S)$ in $\isom_0(AdS_3)$, the 
identity component of the isometry group of $AdS_3$. This identity component 
decomposes as $PSL(2,\R)\times PSL(2,\R)$, so $\rho$ decomposes as 
a pair of representations $(\rho_l, \rho_r)$ of $\pi_1(S)$ in $PSL(2,\R)$.
Mess \cite{mess,mess-notes} proved that those two representations have maximal
Euler number, so that they are holonomy representations of hyperbolic metrics
on $S$. 

He also proved the following theorem, which can be considered as an
AdS analog of the Bers Double Uniformization Theorem for quasifuchsian hyperbolic
manifolds.

\begin{thm}[Mess \cite{mess}] \label{tm:mess}
Given two hyperbolic metrics $h_l,h_r\in \cT$ on $S$, 
there is a unique globally hyperbolic AdS structure on $S\times (0,1)$
such that $\rho_l$ (resp. $\rho_r$) is the holonomy representation of
$h_l$ (resp. $h_r$). 
\end{thm}

We will say that $M$ is {\it Fuchsian} if it contains a totally geodesic, 
closed, space-like, embedded surface. This happens if and only if the 
representations $\rho_l$ and $\rho_r$ are conjugate. 

Let $M$ is a GHM AdS 3-manifold. Then $M$ contains a smallest non-empty
convex subset $C(M)$, called its {\it convex core}. If $M$ is Fuchsian,
then $C(M)$ is the (unique) totally geodesic space-like surface in $M$.
Otherwise, $C(M)$ is topologically the product of $S$ by an interval.
Its boundary is the disjoint union of two space-like surfaces each
homeomorphic to $S$. Those two surfaces are pleated surfaces, their
induced metric is hyperbolic and they are bent along a measured 
lamination. 

There are many future convex surfaces (surfaces for which the 
future is convex) in the past of $C(M)$, and many past convex
surfaces in the future of $C(M)$. Given a strictly future
convex surface $S_c\subset M$, we can define the dual surface (using the 
duality in $AdS_3$, see \ref{subsub:dual}) as the set of points dual to the tangent planes of $S_c$,
it is a past convex surface, and conversely. 

\subsection{Equivariant polyhedral surfaces} \label{subsec: pol fuch}

Consider a globally hyperbolic AdS 3-manifold $M$, and a closed, space-like surface 
$\Sigma$ in $M$. Then $\Sigma$ lifts to an {\it equivariant surface} in the universal
cover of $M$, considered as a subset of $AdS_3$. We recall here the basic definitions
and key facts on those equivariant surfaces, focussing on the polyhedral surfaces
which are relevant below.

\subsubsection{Equivariant embeddings}

Given a closed polyhedral surface $S$ in an AdS 3-manifold $M$, one can consider a 
connected component of the lift of $S$ to the
universal cover of $M$. It is a polyhedral surface invariant under the action of
a surface group. This leads to the definition of an equivariant embedding of a surface
in $AdS_3$, as follows.

\begin{defi}\label{def:eq emb}
An {\it equivariant polyhedral embedding of} $S$ in $AdS_3$ is a couple $(\phi,\rho)$, where:
\begin{itemize}
\item $\phi$ is a locally convex, space-like polyhedral immersion of $\St$ in $AdS_3$,
\item $\rho:\pi_1(S)\rightarrow \isom_0(AdS_3)$ is a homomorphism, 
\end{itemize}
satisfying the following natural compatibility condition:
$$ \forall x\in \St, \forall \gamma\in \pi_1(S), \phi(\gamma.x)=\rho(\gamma)\phi(x)~. $$
An equivariant polyhedral surface is a surface of the form $\phi(\St)$, where $(\phi,\rho)$
is an equivariant polyhedral embedding of $S$ in $AdS_3$.
\end{defi}

The assumption that $\phi(\St)$ is space-like means that it is made of pieces of space-like 
planes, but also that each support plane is space-like. With this assumption
we get the following classical results, see \cite{mess,mess-notes,BBZ07} and references therein.
The \emph{boundary at infinity} $\partial F$ of a surface $F$ of $AdS_3$ is the intersection, in an affine chart,
of the closure of the surface with the boundary at infinity of $AdS_3$.

\begin{lemma}\label{lem: basics}
Let $(\phi,\rho)$ be an equivariant polyhedral embedding of $S$ in $AdS_3$. Then
\begin{enumerate}
 \item  $\phi$ is an  embedding, 
\item each time-like geodesic of $AdS_3$ meets $\phi(\St)$ exactly once,
\item each light-like geodesic of $AdS_3$ meets $\phi(\St)$ at most once,
\item there exists an affine chart which contains $\phi(\St)$,
\item and if a light-like geodesic has one of its endpoints on $\partial \phi(\St)$
then the geodesic does not meet $\phi(\St)$.
\end{enumerate}
\end{lemma}

A key remark, basically obtained by G.~Mess \cite{mess}, is that the holonomy
representation is the ``product'' of the holonomy representations of two 
hyperbolic surfaces.

\begin{prop} \label{pr:mess}
Let $(\phi,\rho)$ be an equivariant polyhedral embedding of $S$ in $AdS_3$,
let $\rho=(\rho_l,\rho_r)\in PSL(2,\R)\times PSL(2,\R)$. Then $\rho_l$
and $\rho_r$ have maximal Euler number, so that they are the holonomy representations
of hyperbolic metrics on $S$.  
\end{prop}

\begin{proof}
Consider the polyhedral surface $\phi(\St)$. It has a free and properly discontinuous cocompact
action of $\rho(\pi_1(S))$
Therefore it follows from Lemma~\ref{lem: basics} that there is a neighborhood $U$ of $\phi(\St)$
in $AdS_3$ on which $\rho(\pi_1(S))$ also acts freely and properly discontinuously \cite[5.22]{BBZ07}. The quotient
$M=U/\rho(\pi_1(S))$ is an AdS 3-manifold which contains a closed space-like surface
(namely, $\phi(\St)/\rho(\pi_1(S))$). 

It follows (see \cite{mess,mess-notes}) that $M$ has an isometric embedding in a
(unique) globally hyperbolic AdS 3-manifold $N$. Moreover $N=\Omega/\rho(\pi_1(S))$,
where $\Omega$ is a maximal convex subset of $AdS_3$ on which $\rho(\pi_1(S))$ 
acts properly discontinuously. Finally, it is shown in \cite{mess} that 
$\rho=(\rho_l,\rho_r)$, where $\rho_l$ and $\rho_r$ have maximal Euler number.
\end{proof}

\subsubsection{Fuchsian equivariant surfaces}

Among the equivariant surfaces, some are simpler to analyze. 

\begin{defi}
An equivariant polyhedral embedding $(\phi,\rho)$ of $S$ in $AdS_3$ is {\it Fuchsian} if its 
representation $\rho$ globally fixes a totally geodesic space-like plane in $AdS_3$.
A {\it Fuchsian polyhedral  surface} is a polyhedral surface of the form $\phi(\St)$, where 
$(\phi,\rho)$ is a Fuchsian equivariant polyhedral embedding.
\end{defi}

Equivalently, $(\phi,\rho)$ is Fuchsian if, in the identification of $\isom_0(AdS_3/\Z_2)$ with 
$PSL(2,\R)\times PSL(2,\R)$, $\rho$ corresponds to a couple 
$(\rho_l,\rho_r)\in PSL(2,\R)\times PSL(2,\R)$ such that $\rho_l$ is conjugate to $\rho_r$.

\subsubsection{Equivariant polyhedral surfaces as convex hulls}

Let $(\phi,\rho)$ be a convex polyhedral embedding of  $S$. 
If $V$ is the set of the vertices of $\phi(\St)$, we denote by $\partial CH(V)$ the boundary of the convex hull of $V$
(there exists an affine chart which contains $\phi(\St)$, see Lemma~\ref{lem: basics},
hence it is meaningful to speak about convex hull). Let $v_1,\cdots,v_n$ 
be a set of representatives of the equivalence classes for the action of
$\rho(\pi_1(S))$ on $V$. The embedding is said to have $n$ fundamental vertices.

\begin{lemma}\label{ref:lem variation}
An equivariant polyhedral embedding $(\phi,\rho)$ of $S$ is determined by $\rho$ and $v_1,\cdots,v_n$.
Moreover any sufficiently small perturbation of the $v_i$ gives another 
equivariant polyhedral embedding of $S$ with $n$ fundamental vertices. 
\end{lemma}
\begin{proof}
It is clear that the data of $\rho$ and $v_1,\cdots,v_n$ gives $V$.

Recall that the \emph{convex core} $C(\rho)$ for $\rho$ 
is the  minimal (for the inclusion) non-empty convex set in $AdS_3$ globally invariant under the action of $\rho$.
If $(\phi,\rho)$ is Fuchsian, the convex core is reduced to a totally geodesic surface, otherwise
it has non-empty interior and two boundary components, which are pleated space-like surfaces, 
isometric to the hyperbolic plane for the induced metric.
The convex core is the convex hull of its boundary at infinity, which is a (topological) circle $c(\rho)$
drawn on the boundary at infinity of $AdS_3$. Actually this boundary at infinity is exactly 
$\partial \phi(\St)$.
Note that $\phi(\St)$ does not meet the interior of $C(\rho)$, otherwise the intersection between $C(\rho)$ and
 the half-spaces delimited by the planes
containing the faces of $\phi(\St)$ would define a convex set invariant under $\rho$ 
and contained in $C(\rho)$, this is impossible.
It follows that $\phi(\St)$ is contained in one side of $C(\rho)$. Let us denote by $\partial C(\rho)_-$ the component of
the boundary of $C(\rho)$ opposite to $\phi(\St)$.

In an affine chart, $\phi(\St)$ and $\partial C(\rho)_-$ are both locally convex surfaces. 
The intersection of the closure of both with $\partial_\infty AdS_3$ is the curve $c(\rho)$.
Since $c(\rho)$ is the boundary at infinity of a space-like surface in $AdS_3$, it is 
a weakly space-like curve in $\partial_\infty AdS_3$, for its conformal Lorentz structure.
(Weakly space-like means that no point of $c(\rho)$ is in the future of another.)
Given a point $x\in c(\rho)$, there is a unique light-like plane $p_x$ in $AdS_3$
through $x$. The intersection of $p_x$ with $\partial_\infty AdS_3$ is the union of the two light-like
geodesics of $\partial_\infty AdS_3$ through $x$ (again for the conformal Lorentz structure on 
$\partial_\infty AdS_3$). Therefore $c(\rho)$, since it is weakly space-like, is
contained in the half-space bounded by $p_x$, and so are $\phi(\St)$ and $\partial C(\rho)_-$.
So $p_x$ is a support plane of the closure of $\phi(\St)\cup \partial C(\rho)_-$.
So this closure is locally convex at each point of $c(\rho)$, and it is therefore
locally convex, and it is therefore a convex surface. 

It follows that 
the elements of $V$ are extreme points for the convex hull of $V$.
As the elements  of $V$ accumulates on
 $c(\rho)$, the boundary of the convex hull of $V$ is made of the union of $\phi(\St)$ and $\partial C(\rho)_-$ plus
 a circle at infinity.

Now perturb slightly the $v_1,\cdots,v_n$, this gives new points 
$v'_1,\cdots,v'_n$. Denote by $V'$ the orbits of the $v'_i$ under $\rho$
and by $\phi'(\St)$ the boundary of the convex hull of $V'$ minus $\partial C(\rho)_-$.
If the $v'_i$ are sufficiently near the $v_i$, all the $v'_i$ are extreme points
of $\phi'(\St)$. 
Moreover $\phi'(\St)$ is also disjoint from
$C(\rho)$. We saw that  $\phi(\St)$ is disjoint from
the interior of $C(\rho)$, but it is also disjoint from its boundary. Otherwise, if it touched it at a point $x$, by
convexity $\phi(\St)$ should contain the geodesic line of the boundary of $C(\rho)$ containing $x$, this is 
impossible as $\phi(\St)$ is a polyhedral surface (in particular with compact faces).
As the action of $\rho(\pi_1(S))$ is cocompact on $\phi(\St)$ and $C(\rho)$, the distance between
both is bounded from below. Hence if $v'_1,\cdots,v'_n$ are sufficiently close to $v_1,\cdots,v_n$, 
$\phi'(\St)$ is made of totally geodesic faces.
It is also space-like (with space-like support planes) as
the fact that two points are on a same space-like geodesic is
also an open condition.

Note finally that if $F$ is a face of the convex hull of $V$, it has a finite number of 
vertices in $V$ by definition of a polyhedral embedding. Moreover, in a projective model,
the boundary at infinity of the space-like plane $P$ containing $F$ 
is disjoint from the intersection with $\partial_\infty AdS_3$ of the closure of $\phi(\St)$.
It follows that all points in $V$ which are not a vertex of $F$ are outside a small
neighborhood of $P$ in the projective model of $AdS_3$.

Now consider a deformation with the $v'_i$ close enough to the $v_i$.
Given three vertices $x,y,z$ of $F$, let $x',y',z'$ be the corresponding points of $V'$, and let
$p'$ be the totally geodesic plane containing $x',y',z'$. For
a small enough deformation, any point of $V$ which is not a vertex of $F$ remains in the
past of $p'$, and so does $\partial_\infty \phi(\St)$. It follows that the face $F$ corresponds,
in the deformed polyhedral surface, to a union of faces having as vertices the vertices of $F$.
In other terms all faces of the deformed polyhedral surface remain compact, with a finite number
of vertices.
\end{proof}

\subsection{The left and right multiplications}

This part is very similar to subsections \ref{sub2.1} and \ref{sub2.2}, so we skip some details.

\subsubsection{Polar duality and multiplication}\label{subsub:dual}

The \emph{polar link} $\mathrm{Lk}(x)$ of a vertex $x$ of a convex space-like polyhedral surface $P$ in $AdS_3$ 
is the convex hyperbolic polygon in $T_xAdS_3$ given by the outward unit time-like normals of the faces of $P$ meeting at $x$. 
The \emph{exterior dihedral angle}
of an edge of $P$ with endpoint $x$ is the length of the corresponding edge of $\mathrm{Lk}(x)$. The
 interior angles of $\mathrm{Lk}(x)$ are $\pi$ minus the interior angles on the faces of $P$ at $x$. 

Let $x\in AdS_3$. We denote by $x^*$ the hyperplane orthogonal to the vector $x\in\mathbb{R}^{2,2}$ for $\langle \cdotp,\cdotp \rangle_2$. 
The intersection between $x^*$ and $AdS_3$ is space-like and has two connected components, which are antipodal. 
We also denote by $x^*$ the  intersection between the hyperplane and $AdS_3$. 
Conversely, if $H$ is a totally geodesic plane in $AdS_3$, we denote by $H^*$ 
the unique point of $AdS_3/\mathbb{Z}_2$ orthogonal as a vector to the hyperplane containing $H$.

The \emph{polar dual} $P^*$  of a convex space-like polyhedral surface $P$ is a convex polyhedral surface 
defined as the intersection of the half-spaces $\{y\in AdS_3 \vert \langle x,y\rangle_2<0 \}$
for each vertex $x$ of $P$. Faces of $P$ are contained in $x^*$. 
Actually this defines two convex polyhedral surfaces, which are identified in
$AdS_3/\mathbb{Z}_2$.
Equivalently  $P^*$ can be defined as the convex hull (in an affine chart) of the $H^*$ 
for $H$ a plane  containing a face of $P$. It follows that $(P^*)^*=P$ and
that the polar link $\mathrm{Lk}(x)$ of a vertex $x$ of $P$ is isometric to the face of $P^*$ 
dual to $x$, and vice-versa. It is easy to see that the dual of a space-like
polyhedral surface is space-like. Moreover if $P$ is a convex polyhedral surface equivariant under the action of $\rho$,
as $\rho$ acts by isometries, $P^*$ is also equivariant under the action of $\rho$.

We will identify the two models of $AdS_3$  through the following isometry:

$$AdS_3\subset (\mathbb{R}^4, \langle \cdotp,\cdotp \rangle_2)\rightarrow (SL(2,\R),-\mathrm{det}),\quad(x_1,x_2,x_3,x_4)\mapsto \begin{pmatrix}
         x_2+x_4 & x_1+x_3 \\
         x_1-x_3 & x_4-x_2
\end{pmatrix}.
$$
 The neutral element as a vector of  $\mathbb{R}^4$ is then $e=(0,0,0,1)$,
and if $y=(y_1,y_2,y_3,y_4)$ then $y^{-1}=(-y_1,-y_2,-y_3,y_4)$.

Let $a,b$ be in $AdS_3/\mathbb{Z}_2$, with $a,b,e$  mutually different, and such that 
$a^*$ and $b^*$ meet along a $2$-plane $E$ in $\mathbb{R}^{2,2}$.
Let $A$ be the following angle in $AdS_3$:
\begin{eqnarray*}
\ A&=&a^*_c\cup b^*_c\cup E \\
\ &=& \{x \in AdS_3/\mathbb{Z}_2\vert \langle x,a\rangle_2=0,\langle x,b\rangle_2 < 0 \}\cup \\
\ &&\{x \in AdS_3/\mathbb{Z}_2\vert \langle x,b\rangle_2=0,\langle x,a\rangle_2 < 0 \}\cup \{x\in AdS_3/\mathbb{Z}_2\vert \langle x,a \rangle_2=\langle x,b \rangle_2=0 \}.
\end{eqnarray*}

The \emph{left projection} of the angle $A$ is the following  map:

$$
L_A :  A\setminus E  \longrightarrow  e^*,
     x   \longmapsto 
\left\{
     \begin{matrix}
   a^{-1}x \mbox{ if } x\in a_c^*    \\
    b^{-1}x \mbox{ if } x\in b_c^*    \\
            \end{matrix}
\right. .
$$

The \textit{right projection}  of the angle $A$ is defined similarly:
$$
R_A :  A\setminus E  \longrightarrow  e^*,
     x   \longmapsto 
\left\{
     \begin{matrix}
   xa^{-1} \mbox{ if } x\in a_c^*    \\
    xb^{-1} \mbox{ if } x\in b_c^*    \\
            \end{matrix}
\right. .
$$

We assume that $R_A$ and $L_A$ have values in $AdS_3/\mathbb{Z}_2$,
so that (the quotient of) $e^*$ is isometric to the hyperbolic plane $H^2$: we consider that $L_A$ and $R_A$ have values in 
$H^2$.
By arguments very similar to the proof of Lemma~\ref{lem: proj angle}
for the spherical case, it is easy to see that the 
left (resp. right) projection sends isometrically the faces of 
the angle to two half-parts of the hyperbolic plane, in such a way that
the part on the left (resp. right) is forward and the part on the right
(resp. left) is backward.

\subsubsection{Equivariant polyhedral surfaces as flippable tilings}\label{subsub pol flip}

The left and right projections of  a space-like convex polyhedral surface $P$ of $AdS_3/\mathbb{Z}_2$
are defined similarly to the spherical case.
Let $(\phi,\rho)$ be an equivariant polyhedral embedding of $S$ with
$\rho=(\rho_r, \rho_l)$ and $\phi(\tilde{S})=P$. The action of 
$\rho(\pi_1(S))$ sends a point $x\in P$ to a point $\rho(\gamma)(x)=\rho_l(\gamma)x\rho_r(\gamma)$ of $P$. 
If $x$ belongs to a face of $P$ contained 
in a plane $a^*$, the left projection of $P$ 
sends $x$ to $a^{-1}x$ and sends $\rho(\gamma)(x)$ to $\rho_r(\gamma)^{-1}a^{-1}x\rho_r(\gamma)$.
If we identify $e^*$ with $H^2$, we obtain an action of 
$\rho_r(\pi_1 S)$ on $H^2$.
By Proposition~\ref{pr:mess}, $H^2/\rho_r(\pi_1 S)$ endows 
$S$ with a hyperbolic metric $g_l$.
We denote by $T_l(P)$ this hyperbolic surface together with the tiling given 
by the images of the faces of $P$. We paint these images 
in white and the remainder of the surface in black, and by arguments 
similar to the proof of Proposition~\ref{prop: proj sph},
we get that $T_l(P)$ is a left flippable tiling of the surface.

Similarly, the right projection of $P$ endows $S$ with a hyperbolic metric $g_r$ together with 
a right flippable tiling denoted by $T_r(P)$. 
Note that  $P$ is Fuchsian if and only if there is an isometry isotopic to the identity between
$g_r$ and $g_l$.

\subsubsection{Flippable tilings as equivariant polyhedral surfaces} 
 
Let $T_l$ be a left flippable tiling of a compact hyperbolic surface $S$.
We denote by $\widetilde{T_l}$ the universal cover of $T_l$. 
As in the spherical case, there exists a locally convex polyhedral
immersion of the white faces of $\widetilde{T_l}$
in $AdS_3$, such that the polar links of the vertices are isometric to the black faces. 
This polyhedral immersion is defined only up to global isometries. We denote by $\phi$ the choice of one immersion.
Let $\phi(x)$ belong to a face of $\phi(\widetilde{T_l})$. There is a 
unique isometry in the connected component of the identity of Isom($AdS_3$)
sending the plane containing this face to the plane containing the face containing $\phi(\gamma x)$,
and sending the vertices of the face containing $\phi(x)$ to the corresponding vertices of 
the face containing $\phi(\gamma x)$.
This provides a representation $\rho$ of $\pi_1(S)$ into $\isom_0(AdS_3)$. Because the polar link
at each vertex is isometric to a compact hyperbolic polygon, all the support planes at each vertex 
are space-like.
It follows that $(\phi,\rho)$ is an equivariant polyhedral embedding of $S$ in $AdS_3$ 
in the sense of Definition~\ref{def:eq emb}.  We call 
it the \emph{white polyhedron} of $T_l$ and denote it by $P_w(T_l)$.

\subsubsection{Proof of Theorem~H\ref{tmH:base}}

By construction, the proof is word by word the same as the one of Theorem~S\ref{tmH:base}, 
considering a closed hyperbolic surface instead of the sphere.

\subsubsection{Proof of Theorem~\ref{tm:earthquake}} \label{sssc:earthquake}

Given $h_l, h_r\in \cT$, there is by Theorem \ref{tm:mess} a unique globally hyperbolic
AdS manifold $M$ with left and right representations $\rho_l$ and $\rho_r$ 
equal to the holonomy representations of $h_l$ and of $h_r$, respectively. 

A convex polyhedral surface $\Sigma$ in $M$ determines an equivariant polyhedral surface
$\tilde{\Sigma}$ with associated representation $\rho=(\rho_l,\rho_r)$.
According to the content of Section \ref{subsub pol flip}, $\tilde{\Sigma}$ determines
in turn a left flippable tiling $T_l$ on $(S,h_l)$ such that flipping $T_l$ yields
a right flippable tiling $T_r$ on $(S,h_r)$. So the proof of Theorem \ref{tm:earthquake}
reduces to showing that $M$ contains a convex polyhedral surface. 

However the existence of such a polyhedral surface is quite easy to check. Given a
globally hyperbolic AdS manifold $M$, consider a point $x_1$ in the future of the
convex core $C(M)$. Since $C(M)$ is convex, there is a totally geodesic plane 
$P_1$ passing through $x_1$ and in the future of $C(M)$, and let $H_1$ be the 
future of $P_1$. Then repeat this procedure with a point $x_2$ which is not in 
$H_1$, then with a point $x_2$ not in $H_1\cup H_2$, etc. After a finite number
of steps, the boundary of the intersection of the complements of the $H_i$
is a closed, convex polyhedral surface in $M$.

\subsubsection{Fuchsian polyhedral surfaces and symmetric tilings}\label{subsub fuchs sym}
 
We have already noticed at the end of Subsection~\ref{subsub pol flip}
that  $P$ is Fuchsian if and only if there is an isometry isotopic to the identity
sending $T_r(P)$ to $T_l(P)$. As $T_l(P)$ is obtained by flipping $T_r(P)$, it follows that
if $P$ is Fuchsian then $T_r(P)$ is symmetric in the sense of Definition~\ref{def sym til}.
Conversely, let $T_r$ be a symmetric right flippable tiling  of a surface $(S,g_r)$.
This means that $T_r'=T_l(P_w(T_r))$ is isometric to $T_r=T_r(P_w(T_r))$, i.e.~that
$P_w(T_r)$ is Fuchsian.

\subsection{Proof of Theorem~H\ref{tm:moduli-h}} 

Let $\Gamma$ be a graph embedded in a closed surface $S$ of hyperbolic type,
such that the universal cover of $\Gamma$ is 3-connected.
We will prove that $\cT^{r,-1}_{S,n}(\Gamma)$ is non-empty. Let us denote by $\Gamma^*$ the dual graph of
$\Gamma$.
Recall Thurston's extension to hyperbolic surfaces of Koebe's circle packing theorem
(see e.g. \cite{Schpoly,Sch03}): there exists a unique hyperbolic metric $h$ on $S$ and a unique circle packing in 
$(S,h)$ such that:
\begin{itemize}
\item there is a circle for each vertex of $\Gamma^*$, and the circles bound non-intersecting open discs,
\item the circles are tangent if the corresponding vertices are joined by an edge of $\Gamma^*$,
\item faces of $\Gamma^*$ are associated to interstices (connected components of $S$ minus the open discs), 
\item and to each interstice 
is associated a circle orthogonal to all the circles corresponding to the vertices of this face.
\end{itemize}

Let $H$ be a totally geodesic space-like plane in $AdS_3$. It is isometric to
the hyperbolic plane $H^2$, and can be identified isometrically with the universal
cover of $(S,h)$. This yields a circle packing $\cC$ with incidence graph $\Gamma^*$
on the quotient of $H$ by an isometric action $\rho$ of $\pi_1(S)$ in the isometry
group of $H$. We extend $\rho$ to an isometric action on $AdS_3$.

Let $r\in (0,\pi/2)$, and let $H_r$ be the set of points at distance $r$ in 
the future of $H$. For each circle $C$ of $\cC$ corresponding to a vertex of $\Gamma^*$, consider the set $C'$ of 
points of $H_r$ which project orthogonally on $H$ to a point of $C$. $C'$ is
the intersection with $H_r$ of a totally geodesic space-like plane $H_c$,
let $P_C$ be the past of $H_c$. The intersection of the half-spaces $P_C$
for all $C\in \cC$ is a convex polyhedron in $AdS_3$, invariant under $\rho(\pi_1(S))$,
and its boundary $S_\cC$ is a space-like polyhedral surface in $AdS_3$, invariant under
$\rho$. 

This equivariant polyhedral surface $S_\cC$ has the same combinatorics as $\Gamma$. 
To see this, it suffices to note that for all circles of $H_r$ corresponding to faces around a same vertex of $\Gamma$, the corresponding planes
meet at a same point, namely the apex of the cone tangent to $H_r$ along the image of the circle orthogonal to the previous ones.
(This is a projective property. It is easy to see if one considers $H$ in the hyperbolic space, where $H_r$ is a part of the boundary at infinity. 
Then the projection from $H$ to $H_r$ preserves orthogonality, and the result follows using the polar duality.)

Through the content of Section \ref{subsub pol flip},
we obtain a flippable tiling on $(S,h)$ with incidence graph $\Gamma$, this
proves that $\cT^{r,-1}_{S,n}(\Gamma)$ is non-empty. This proves point (1).

For point (3) note that the space of flippable tilings with $n$ black faces
is homeomorphic, by Lemma~\ref{lm:topologies}, to the space of equivariant polyhedral 
surfaces with $n$ vertices. By Lemma~\ref{ref:lem variation}, this space is parameterized by the holonomy 
representation of the fundamental group of the surface in the isometry group 
of $AdS_3$ and by the positions of the vertices, so that it is homeomorphic to
a manifold of dimension $12g-12+3v$.

We now turn to the proof of part (2), and consider a fixed graph $\Gamma$.
By Lemma \ref{lm:topologies}, it is sufficient to prove that the space of
equivariant polyhedral surfaces in $AdS_3$, with combinatorics given by
$\Gamma$, is homeomorphic to a manifold of dimension $6g-6+e$.

Denote by $n_f$ the sum, over all faces
of $\Gamma$, of the number of edges minus $3$. We get
$$  n_f = 2e-3f~. $$
As the combinatorics is given, the vertices belonging to a same face have to stay in the same totally geodesic plane during the deformation. 
For each face, given three vertices of the face, the other vertices of the face have to belong to the affine plane (in the projective model of $AdS_3$)
spanned by the three vertices. 
This gives $n_f$ equations, so $\cT^{r,-1}_{S,n}(\Gamma)$ is homeomorphic to a manifold of
dimension $12g-12+3n-n_f$. Using the Euler relation $n-e+f=2-2g$ we obtain that the dimension is also equal to
$6g-6+e$. 

Finally point (4) is a direct consequence of the constructions, since the 
mapping-class group of $S$ clearly acts on the space of flippable tilings.

\subsection{Proof of Theorem~H\ref{tm:symmetric}}
\label{sec:proof sym}

The proof is an immediate adaptation of the proof of Theorem~S\ref{tm:sphere} in Subsection~\ref{sub: proofS4}, using the following result
instead of Alexandrov Theorem (Theorem~\ref{thm: alexandrov}).

\begin{thm}[{\cite{Fil09}}]\label{thm: pol fuchs}
 Let $(S,g)$ be a hyperbolic metric with conical singularities of negative curvature on a closed surface $g$ of genus $>1$. 
Then there exists a Fuchsian equivariant polyhedral embedding $(\phi,\rho)$ of $S$ such that the induced metric on $\partial \phi(\St)/\pi_1(S)$ is isometric to $(S,g)$. 
Moreover, up to global isometries, $( \phi,\rho)$
 is unique among Fuchsian equivariant polyhedral embedding.
\end{thm}

\section{Alexandrov and Minkowski type results in AdS}
\label{sc:alexandrov}

In this section, we will prove a AdS analogue of a standard result of Euclidean geometry, known as 
Alexandrov curvature theorem, see for example \cite[9.1]{Ale05} or \cite{Pakbook}. Roughly speaking, the problem
is to prescribe geodesics containing vertices of a convex polyhedron together with the singular curvature at the vertices. 
It will appear that it is equivalent to an analogue of the Minkowski theorem for Euclidean convex polytopes.

\subsection{Two equivalent statements}

Let $(\phi,\rho)$ be a Fuchsian equivariant polyhedral embedding of a surface $S$ in $AdS_3$,
and let $G=\rho(\pi_1(S))$. 
By definition $G$ fixes a totally geodesic space-like plane $H$ in $AdS_3$, and it follows from the proof of Lemma~\ref{ref:lem variation} that
$\phi(\St)$ does not meet $H$.
Without loss of generality, suppose that $\phi(\St)$ is in the future of $H$.
 $G$ also fixes the dual $H^*$ of $H$ 
($H^*$ is chosen to be in the future of $H$). The boundary at infinity $\partial H$ of $H$ is also the boundary at infinity
of $\phi(\St)$, and all light-like geodesics passing through a point of $\partial H$ also pass through $H^*$.
By Lemma~\ref{lem: basics} it follows that $\phi(\St)$ is contained in the past cone of $H^*$, and that each geodesic orthogonal
to $H$ (i.e.~with endpoint $H^*$) meets $\phi(\St)$ exactly once.

Let $GR$ be a set of  future-directed times-like segments of length $\pi/2$ orthogonal to $H$, 
 invariant under the action of $G$. The elements of $GR$ are called \emph{half-rays}.
 We denote by $R=(r_1,\cdots,r_n), n\geq 1$ a subset of $GR$ in a fundamental domain for the action of $G$.
Let
$\mathcal{P}(G, R)$ be the set, up to global isometries,  of  
Fuchsian polyhedral surfaces invariant under the action of $G$ 
and such that the vertices  are on $GR$. 
Recall that the 
(singular) curvature of a vertex  is $2\pi$ minus the sum of the face angles around this vertex.

\begin{thm}\label{thm: presc curvADS}
  If $k_1,\cdots,k_n$ are negative numbers such that
\begin{equation}\label{eq: GB}\sum_{i=1}^n k_i>2\pi \chi(g) \end{equation}
then there exists a unique Fuchsian polyhedral surface in $\mathcal{P}(G, R)$
such that the vertex  on $r_i$ has curvature $k_i$.
\end{thm}

\begin{thm}\label{thm: minkADS}
 Under the assumptions of Theorem~\ref{thm: presc curvADS}, there exists a unique
 Fuchsian polyhedral surface with faces contained in planes orthogonal to $G R$, such that  the face in the 
plane orthogonal to $r_i$ has area  $-k_i$. 
\end{thm}

By duality the two statements are equivalent:
\begin{lemma}\label{lem: eq stat}
 If $P$ is a Fuchsian polyhedral surface given by one of the two theorems above, then the dual $P^*$ of $P$ satisfies the other statement. 
\end{lemma}
\begin{proof}
Let $P$ be the unique Fuchsian polyhedral surface satisfying Theorem~\ref{thm: presc curvADS} for the data $(r_1,\cdots,r_n)$ and $(k_1,\cdots,k_n)$.
Let $x$ be a vertex of $P$ on the half-ray $r_i$. 
By definition of the polar dual, the plane $x^*$ is orthogonal to each time-like geodesic passing through $x$. In particular, the geodesic containing
$r_i$ is orthogonal to $x^*$. More precisely, $x^*$ is orthogonal to a half-ray $\tilde{r}_i$ obtained from $r_i$
by a symmetry with respect to $H$.
By definition, $P^*$ is the intersection of all the $x_i^*$ for each vertex $x_i$ of $P$, 
and by the Gauss--Bonnet Formula and by definition of the polar link, the area of the corresponding face of $P^*$ is  $-k_i$.
Hence $P^*$ satisfies the existence part of Theorem~\ref{thm: minkADS} for the data  $(\tilde{r}_1,\cdots,\tilde{r}_n)$ and $(k_1,\cdots,k_n)$.

The reversed direction for the existence part is done in the same way, and this implies the equivalence of
the uniqueness parts.
 
\end{proof}

Note also that condition (\ref{eq: GB}) is necessary in Theorem~\ref{thm: presc curvADS} because of  Gauss--Bonnet formula, and by the preceding lemma it is also necessary in Theorem~\ref{thm: minkADS}. 
Conversely to the Euclidean Alexandrov curvature theorem, there is no other condition on the curvatures.
Actually Theorem~\ref{thm: presc curvADS} can be though as the AdS analog of the Alexandrov curvature theorem for convex caps, in which Gauss--Bonnet is the only condition on the curvatures \cite{Pakbook}.

The analog of Theorem~\ref{thm: presc curvADS} for Fuchsian polyhedral surfaces in Minkowski space was proved in \cite{Isk00}. 
It was extended to general Fuchsian surfaces in this space, and to higher dimensions, in \cite{Ber10}.

\subsection{Proof of Theorem~H\ref{thm: paramADS} from Theorem~\ref{thm: presc curvADS}}

Let $T_r$ be a symmetric flippable tiling of $\cbT^{r,-1}_{S,h,n}(K)$.
From Subsubsection~\ref{subsub fuchs sym} we can choose a Fuchsian polyhedral surface $P$ constructed from $T_r$, such that
the singular curvatures of $P$  (i.e.~ minus the area of the polar links of its vertices) are equal to the area of the black faces of $T_r$.
This gives us a group $G \subset \isom(AdS_3)$ fixing a totally geodesic space-like plane $H$. 
Hence $\cbT^{r,-1}_{S,h,n}(K)$ is in bijection with $\mathcal{P}(G, R)$.
Theorem~\ref{thm: presc curvADS} says that $\mathcal{P}(G, R)$ is parameterized by  the orthogonal projection
of the vertices onto $H$, which is the same as the configuration of $n$ points on $H/G \simeq S$. This is exactly
the content of Theorem~\ref{thm: presc curvADS}.

\subsection{Proof of Theorem~\ref{thm: presc curvADS}}

The proof uses a classical continuity method.
We recall here a basic topology theorem at the heart of this continuity method.

\begin{lemma}\label{lem: cont meth}
 Let $A, B$ be two manifolds of the same finite dimension with a continuous map $I : A \rightarrow B$. If
\begin{enumerate}
\item $A$ is connected, 
\item $B$ is connected and simply connected, 
\item $I$ is locally injective,
\item $I$ is proper,
\end{enumerate}
then $I$ is a homeomorphism.
\end{lemma}
\begin{proof}
 As $I$ is continuous and locally injective, it is a local homeomorphism by the invariance of domain theorem. 
The conclusion follows because a local homeomorphism between a pathwise connected Hausdorff space and a simply connected Hausdorff space is a global homeomorphism if and only if the map is proper (see for example \cite{Ho75}).
 \end{proof}

 Let $P\in\mathcal{P}(G, R)$. The  \emph{height} $h_i$ of $P$ is  
the distance from its  vertex on $r_i$ to the hyperbolic plane $H$. 
We know from Lemma~\ref{ref:lem variation} that an element of $\mathcal{P}(G, R)$ is determined by its  heights 
$(h_1,\cdots,h_n)$,  and that $\mathcal{P}(G, R)$ is a manifold of dimension $n$. 
We need a slightly more here. The matter is simplified because in the Fuchsian case the convex core has empty interior.
\begin{lemma}
Let $(h_1,\cdots,h_n)$, $0<h_1<\pi/2$. If the orbits for $G$ 
of the points on the  segments $r_i$ at distance $h_i$ from $H$
are in convex position, then the convex hull of these orbits has two connected components, 
one is $H$ and the other belongs to  $\mathcal{P}(G, R)$.

Moreover $\mathcal{P}(G, R)$ is contractible.
\end{lemma}
\begin{proof}
Let $P$ be the convex hull of the orbits less $H$. By construction $P$ in invariant under the action of $G$
and has its vertices on $GR$. To prove that it belongs to $\mathcal{P}(G, R)$, it remains to check that it is a space-like
polyhedral surface. 

Let us consider the case when $P$ is made only of the orbit of one vertex.
All the vertices of $P$ belongs to a space-like surface at constant distance from $H$. By 
Lemma~\ref{lem: basics}, this surface meets at most once every time-like and light-like geodesic
hence $P$ is space-like. The fact that the support planes are space-like comes from the 
cocompactness of the action of $G$ on $H$, the argument is the same as for Fuchsian surfaces in the Minkowski space \cite{FF}.
Now consider the surface $Q$ which is the boundary of the intersection of half-spaces bounded by 
planes orthogonal to the half-rays at the vertices of $P$. It is not hard to see that 
the orthogonal projection of the 1-skeleton of $Q$ onto $H$ is a Dirichlet tessellation of 
$H$ for the group $G$, hence $Q$ is a polyhedral surface as $G$ acts cocompactly on $H$.
The dual of $Q$ is a polyhedral surface, equal to $P$, up to translate the vertices along the $r_i$.

The case with more one orbit follows as in this case the surface can be seen as the boundary of 
the intersection of the convex hull of each orbit.

Hence $\mathcal{P}(G, R)$ can be identified with the set of heights 
such that the vertices are in convex position, and it is not hard
to see that it is homeomorphic to a convex subset of $\mathbb{R}^n$ (see \cite{Fil09}
where a very similar argument is used).
\end{proof}

Let us define  $\mathcal{K}(n)=\{(k_1,\cdots,k_n)\in (\mathbb{R}_-)^n \vert \sum k_i > 2\pi \chi(S) \}$. 
It is obviously a non-empty contractible (convex) open subset of $\mathbb{R}^n$. 

We have a natural map which to each Fuchsian polyhedral surface associates its singular curvatures: 
$$\begin{array}{lrcl}
C : & \mathcal{P}(G,R) & \longrightarrow & \mathcal{K}(n)\\
    & (h_1,\cdots,h_n) & \longmapsto & (k_1,\cdots,k_n)\end{array}.$$

Theorem~\ref{thm: presc curvADS} says exactly that $C$ is a bijective map. 
We will prove that $C$ 
is  proper (Lemma~\ref{lem: propr}) and locally injective (Lemma~\ref{lem : jac}). Then by Lemma~\ref{lem: cont meth}, 
$C$  is a homeomorphism, in particular a bijective map. 

\begin{lemma}\label{lem: propr}
The map $C$ is proper: 
Let $(K_i)_i$ be a converging sequence of  $\mathcal{K}(n)$ such that for all $i$, 
there exists $P_i\in \mathcal{P}(G,n)$ with $K_i=C(P_i)$. 
Then a subsequence of $(P_i)_i$ converges in $\mathcal{P}(G,n)$.

 \end{lemma}
\begin{proof}

We will first prove that the sequence of heights defining the polyhedra $P_i$ has a converging subsequence
in $]0,\pi/2[^n$. Actually, the  sequence of heights belongs to $[0,\pi/2]^n$ which is compact, 
so $(P_i)_i$ has a converging subsequence in this set (that we again denote by $(P_i)_i$), 
hence it
suffices to check that no height goes to $0$ or to $\pi/2$.

First suppose that all the heights converge to $0$. In this case $(P_i)_i$ converges to a degenerated polyhedron (the hyperbolic
plane $H$). This means that all the curvatures go to zero, that is impossible, as the sequence of
curvatures is supposed to converge in $\mathcal{K}(n)$. So there exists at least one half-ray
$x$ such that the sequence of heights on this half-ray doesn't go to $0$. If the heights on another
ray $y$ ($y\notin G x$) go to zero, then, if  $i$ is sufficiently large, the vertex of $P_i$ on $y$ has
to be inside the convex hull of the points $G x_i$, where $x_i$ is the vertex of $P_i$ lying on $x$, 
this contradicts the fact that $P_i$ is convex.

Now suppose that all the heights converge to $\pi/2$. This means that all the vertices converge to 
the point $H^*$  dual of the hyperbolic plane $H$. In particular, 
the face areas in a fundamental domain for $G$ go to zero, and by the Gauss--Bonnet formula,
the sum of the curvatures go to $2\pi\chi(g)$. But this is impossible as by assumption 
the sequence of curvatures converges inside  $\mathcal{K}(n)$.

Therefore there exists at least one half-ray
$x$ such that the sequence of heights on this ray doesn't go to $\pi/2$. For each $i$, let $x_i$ be the vertex
of $P_i$ on $x$. As $P_i$ is convex, it is contained between the plane $H$ 
and the cone of vertex $x_i$ formed by the planes
containing the faces of $P_i$ meeting at $x$. 
Each half-ray containing a vertex meets
those planes, and this prevents the other heights from going to $\pi/2$. 
Indeed, if  $L$ is a plane containing a face of $P_i$, as it is space-like, 
then it meets every time-like geodesic. Moreover the orbit of 
$x$ accumulates on the boundary at infinity of $H$, so $L$ does not meet $H$, 
as $P_i$ is on one side of $L$. It is then clear that every time-like geodesic from $H$ to $H^*$ has to meet $L$.  

We proved that the heights of $(P_i)_i$ converge in $]0,\pi/2[$. Now as each $P_i$ is a convex polyhedron and
as the sequence of induced curvatures converges, it is clear that  $(P_i)_i$  converges inside $\mathcal{P}(G,n)$.
\end{proof}

\begin{lemma}\label{lem : jac}
The map $C$ is locally injective.
\end{lemma}

\begin{proof}
The idea is to study the Jacobian of $C$ and to invoke the local inverse theorem. Unfortunately, it is not clear
whether $C$ is $C^1$, as the combinatorics of the polyhedra may change under deformation. Let $P\in \mathcal{P}(G,n)$. 
Triangulate the faces of $P$ such that 
no new vertex arises and such that the triangulation is invariant under the action of $G$. 
The new edges added are called \emph{false edges} and the older ones \emph{true edges}.
We denote by
$\Gamma$  the combinatorics of the resulting triangulation 
and by $\mathcal{P}^{\Gamma}(G,R)$ the set of  polyhedral surfaces 
invariant under the action of $G$
with  vertices on $GR$ with a triangulation of combinatorics $\Gamma$
(the elements of $\mathcal{P}^{\Gamma}(G,R)$ are not asked to be convex). We denote by $C_{\Gamma}$ the map from 
$\mathcal{P}^{\Gamma}(G,n)$  to $\mathbb{R}^n$ which associates to each polyhedral surface the curvature at the vertices.
 In Subsection~\ref{subsec: jac} we will prove that, 
at a point $P\in\mathcal{P}(G,R)$, $C_{\Gamma}$ is $C^1$ and its Jacobian  
is non-degenerate. (Actually we will study the map which associates  to the heights the cone-angles, 
instead of curvatures, this does not change the result.) Hence $C_{\Gamma}$ is locally injective around $P$.

Now suppose that $C$ is not locally injective. 
This means that there exists $P'\in\mathcal{P}(G,R)$,
arbitrarily close to $P$, with the same curvature at the corresponding vertices.
As $P'$ is sufficiently close 
to $P$, they can be endowed with a triangulation of same combinatorics $\Gamma$, such that, if new vertices appear, 
they appear on a true edge.
We don't consider the heights of the false new vertices as variables, and
we don't compute the possible variation of the cone-angles
(which are equal to $2\pi)$ at them. Only the angles on the faces around the false vertices
enter the  computations. So the corresponding map
  $C_{\Gamma}$ is not locally injective, that is a contradiction.
\end{proof}

\subsection{Interlude: The case of the sphere}

In this section we switch from the geometry of AdS to that of the sphere. Although the
results presented here are of independent interest, our main motivation is to present
the (rather technical) arguments of the next section in a context in which most readers
are more confortable. Following the same arguments in the AdS setting should then be
relatively easy.

We wonder whether there exists an analogue of Theorem~\ref{thm: presc curvADS} for the sphere. 
As in the AdS case, the question can be reduced to proving
that the map sending the heights $h_x$ determining a convex polyhedron to 
the curvatures $k_x$ (actually the cone-angle $\omega_x$) at its vertices is a bijection (between suitable spaces)---
 here the half-rays should be from a point $o$ contained in the interior of the polyhedron. 
We will compute the Jacobian of this map --- in a simplified  case where
 we assume that the polyhedron has only triangular faces.

We need the following lemma of spherical trigonometry.

\begin{lemma}\label{lem:trigoS}
 Let $a,b,c$ be the side lengths of a spherical triangle with angles $\alpha, \beta, \gamma$. 
Considering $b$ as a function of $a,c,\beta$, we get
\begin{equation}
\frac{\partial b}{\partial a}= \cos\gamma \label{eq: trigo sph 2}.
\end{equation}
Considering $\alpha$ as a function of $a,c,\beta$, we get
\begin{equation}
\frac{\partial \alpha}{\partial a}=\frac{\sin \gamma}{\sin b}  \label{eq: trigo sph 1},
\end{equation}
\begin{equation}
\frac{\partial \alpha}{\partial c}= -\frac{\sin\alpha\cos b}{\sin b}  \label{eq: trigo sph 3}.
\end{equation}
\end{lemma}
\begin{proof}
The derivative of the cosine law
\begin{equation}\label{eq: cos law1}
 \cos b = \cos c \cos a + \sin c \sin a \cos \beta
\end{equation}
gives
\begin{eqnarray*}
 \ \frac{\partial b}{\partial a}&=&\frac{\sin a \cos c - \cos a \sin c \cos \beta}{\sin b} \\
\ &\stackrel{(\ref{eq: cos law1})}{=}&\frac{\cos c - \cos b \cos a}{\sin a \sin b}\stackrel{*}{=}\frac{\sin a \sin b \cos \gamma}{\sin a \sin b}=\cos \gamma,
\end{eqnarray*}
and (\ref{eq: trigo sph 2}) is proved (we used  another cosine law at $*$).

Now we consider $\alpha$ as a function of the side-lengths $a,b,c$, and we denote it by $\tilde{\alpha}$ to avoid confusion with $\alpha$ appearing in  (\ref{eq: trigo sph 1}) and (\ref{eq: trigo sph 3}). The derivative of the suitable cosine law gives that
\begin{equation}\label{eq: cos int 1}
 \frac{\partial \tilde{\alpha}}{\partial a}=\frac{1}{\sin b \sin \gamma}.
\end{equation}
and that 
\begin{equation}\label{eq: cos int 2}
 \frac{\partial \tilde{\alpha}}{\partial b}=-\frac{\cos c -\cos a \cos b}{\sin^2 b \sin c \sin \alpha}
\stackrel{**}{=}-\frac{\sin c\cos \gamma}{\sin b \sin c \sin \gamma}=-\frac{\mathrm{cotan} \gamma}{\sin b}
\end{equation}
(we used another cosine law and a sine law at $(**)$). Now if $\alpha$ is a function of $a,c,\beta$ we have
\begin{equation}\label{eq: derivee}
  \frac{\partial \alpha}{\partial a}=\frac{\partial \tilde{\alpha}}{\partial a}+\frac{\partial \tilde{\alpha}}{\partial b}\frac{\partial b}{\partial a} 
\end{equation}
and (\ref{eq: trigo sph 1}) follows by using (\ref{eq: cos int 1}), (\ref{eq: cos int 2}) and  (\ref{eq: trigo sph 2}) in the equation above. Equation (\ref{eq: trigo sph 3}) is proved similarly.

\end{proof}

\begin{figure}[ht]
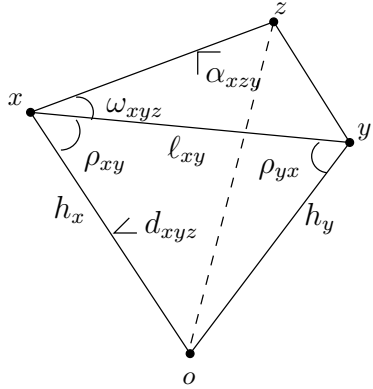
\caption{Notations in a pyramid.\label{fig:pyramid}}
\end{figure}

Let $\omega_x$ be the cone-angle at the vertex on the ray $r_x$.  For each face triangle $xyz$ of  $P$ we get a pyramid $oxyz$. 
The cone-angle $\omega_x$ is decomposed as a sum of angles of the form $\omega_{xyz}$ (see notations on Figure~\ref{fig:pyramid}). 
The link of the pyramid at $x$ is a triangle in which $\omega_{xyz}$ is a function of $\rho_{xy}, \rho_{xz}$ and $d_{xyz}$  when a height varies. Note that 
$d_{xyz}$ remains  fixed by hypothesis.

We denote $a_{xy}:=\frac{\partial \omega_x}{\partial h_y}$. There is at most one edge between two different vertices.
If there is no edge then $a_{xy}=0$ ($y\not= x$).
If there is one edge $xy$, shared by the triangles $xyz$ and $xyz'$, then

\begin{eqnarray*}
 \ && a_{xy}=\frac{\partial \omega_x}{\partial h_y}= \frac{\partial \omega_x}{\partial \rho_{xy}}\frac{\partial \rho_{xy}}{\partial h_y} = \left(\frac{\partial \omega_{xyz}}{\partial \rho_{xy}}+\frac{\partial \omega_{xyz'}}{\partial \rho_{xy}}\right)\frac{\partial \rho_{xy}}{\partial h_y}  \stackrel{(\ref{eq: trigo sph 2}),(\ref{eq: trigo sph 1})}{=}\left(\cos\alpha_{xyz}+\cos\alpha_{xyz'}\right)\frac{\sin \rho_{yx}}{\sin\ell_{xy}} 
\end{eqnarray*}

By convexity $\alpha_{xyz}$ and $\alpha_{xyz'}$ are both positive and $\alpha_{xyz}+\alpha_{xyz'}<\pi $, so  $a_{xy}>0$. 

\begin{remark}
The fact that $a_{xy}>0$  has a clear geometric meaning: 
consider the convex spherical polygon spanned by the edges at $x$. If the height of $y$ increase, we get another convex
spherical polygon which is strictly greater than the former one. As they are convex, the perimeter of the 
new polygon is greater than the former one, and the perimeter is exactly the cone-angle at $x$.

This also means that if the support number $h_y^*$ of the dual of $P$ increases, then  the area
of the face dual to $x$ increases (recall that $h^*_y=\pi/2-h_y$).
\end{remark}

We can also compute the diagonal terms. $E$ is the set of edges of the triangulation, and we imply that
$\ell_{xy}=0$ if there is no edge between $x$ and $y$.

\begin{eqnarray*}
 \ a_{xx}:=\frac{\partial \omega_x}{\partial h_x}=\sum_{xy\in E} \frac{\partial \omega_x}{\partial \rho_{xy}}\frac{\partial \rho_{xy}}{\partial h_x}  \stackrel{(\ref{eq: trigo sph 1}),(\ref{eq: trigo sph 3})}{=}-\sum_{xy\in E}\cos \ell_{xy}\frac{\sin \rho_{xy}}{\sin \rho_{yx}}\frac{\partial \omega_x}{\partial \rho_{xy}}\frac{\partial \rho_{xy}}{\partial h_y} \\
\ =-\sum_{xy\in E}\cos \ell_{xy}\frac{\sin \rho_{xy}}{\sin \rho_{yx}}\frac{\partial \omega_x}{\partial h_y} =-\sum_{x\not = y}\cos \ell_{xy}\frac{\sin \rho_{xy}}{\sin \rho_{yx}} a_{xy}=-\sum_{x\not = y}\cos \ell_{xy}a_{yx}.
\end{eqnarray*}

\begin{remark}
This formula has also a clear geometric meaning. If a face $x^*$ of the dual of $P$ has all its dihedral angles 
$<\pi/2$ (resp. $>$), then if $h_x^*$ increases,   the area of $x^*$ decreases (resp. increases). The case
when the dihedral angles are $\pi /2$ is a kind of critical point.
\end{remark}

It is not clear  if the matrix $(a_{xy})_{xy}$ is non-degenerate. Further study would lead us too far from the scope of this paper. So we address the following question:

\begin{ques}\label{qu}
Is there an analog of the Alexandrov prescribed curvature theorem in the sphere?
More precisely, let $r_1,\cdots, r_n$ be half-rays in the sphere $S^3$ starting from a same point  and $(k_1,\cdots,k_n)$ 
real positive numbers. Does there exist conditions on the $r_i$ and $k_i$ such that the following holds: there exists a unique
convex polyhedron having its vertices on the rays $r_1,\cdots,r_n$, with curvature $k_i$ at the vertex on $r_i$?
\end{ques}
 
An obvious condition on the $k_i$ is that they have to satisfy the Gauss--Bonnet formula.
Note that an answer to the question would give information about the existence of an analog of 
the Minkowski theorem in the sphere (from an argument very similar to the proof of Lemma~\ref{lem: eq stat}).
The corresponding statement in the hyperbolic case is stated without proof in \cite[9.3.1]{Ale05}.

\subsection{End of the proof of Lemma~\ref{lem : jac}}\label{subsec: jac}

\subsubsection{Hyperbolic-de Sitter trigonometry}

We denote by $\mathbb{R}^{n,1}$ the Minkowski space of dimension $n$, 
that is, the space $\mathbb{R}^{n+1}$ endowed with the bilinear form
\begin{equation}
\langle x,y \rangle_1 = x_1y_1+\cdots+x_{n-1}y_{n-1}-x_{n+1}y_{n+1}. \nonumber
\end{equation}
We denote by $\norm{.}_1$ the associated pseudo-norm: $\norm{x}_1:=\sqrt{\langle x,x\rangle_1}$. The pseudo-norm of time-like vectors is chosen as a positive multiple of $i$. A model of the hyperbolic space is the upper-sheet of the two-sheeted hyperboloid:
\begin{equation}
H^n=\{ x\in \mathbb{R}^{n,1} \vert \norm{x}_1^2=-1,\; x_{n+1}>0\};\nonumber
\end{equation}
and a model of the de Sitter space is the one-sheeted hyperboloid:
 \begin{equation}
dS^n=\{ x\in \mathbb{R}^{n,1} \vert \norm{x}_1^2=1\}.\nonumber
\end{equation}
We define the ``de Sitter-Hyperbolic space'' $HS^n$ as the union of the two spaces, and we denote by $d$ its pseudo-distance. It is well-known that
\begin{eqnarray*}
  \ \langle x,y\rangle_1 &=& - \cosh (d(x,y)), (x,y)\in H^n\times H^n,\\
\ \langle x,y\rangle_1 &=& \cos (d(x,y)), (x,y)\in dS^n\times dS^n, \\
\ \langle x,y\rangle_1 &=&   \sinh (d(x,y)), (x,y)\in H^n\times dS^n,
\end{eqnarray*}
where in the last equation the distance between  $x\in H^n$ and  $y\in dS^n$ is defined 
as follows. Let $y^*$ be the hyperplane orthogonal to $y$ for  $\langle \cdot,\cdot \rangle_1$. 
We also denote by $y^*$ its intersection with the hyperbolic space, that is a totally geodesic hypersurface. The distance $d(x,y)$ is  the oriented hyperbolic distance between $x$ and $y^*$: it is positive if  $x$ is lying in the same half-space delimited by $y^*$ than $y$, and negative otherwise. See \cite{Thu97,Sch98} for more details.

Similarly to the Euclidean case, we define the \emph{angle} $\theta$ between two vectors of the Minkowski space as the corresponding distance  on the hyperboloids:
\begin{eqnarray*}
\ &&\cosh (\theta)= \frac{\langle x,y\rangle_1}{\norm{x}_1\norm{y}_1}, (\norm{x}_1,\norm{y}_1)\in i\mathbb{R}_+\times i\mathbb{R}_+, x_{n+1}>0, y_{n+1}>0,\\
\ &&\cos (\theta)= \frac{\langle x,y\rangle_1}{\norm{x}_1\norm{y}_1}, (\norm{x}_1,\norm{y}_1)\in \mathbb{R}_+\times \mathbb{R}_+,\\
\ &&\sinh (\theta) =  i \frac{\langle x, y \rangle_1}{\norm{x}_1\norm{y}_1}, (\norm{x}_1,\norm{y}_1)\in i\mathbb{R}_+\times \mathbb{R}_+.
\end{eqnarray*}
In the first and third cases the angle $\theta$ is a real number.  For the second case we need to describe the different possibilities:
\begin{enumerate}
 \item[-] if $x$ and $y$ span a space-like plane, $\theta\in(0,\pi)$;
\item[-] if $x$ and $y$ span a time-like plane  and intersect the same connected component of the resulting geodesic on the de Sitter space, then $\theta\in i \mathbb{R}_+$;
\item[-] if $x$ and $y$ span a time-like plane  and intersect different connected components of the resulting geodesic on the de Sitter space, then $\theta\in \pi -i \mathbb{R}_+$ (this follows from the preceding case as by definition the cosine of the angle between $x$ and $y$ is minus the cosine of the angle between $-x$ and $y$).
\end{enumerate}

 It is also possible to define this angle if $x$ and $y$ belong to different 
sheets of the hyperboloid with two sheets, but we won't use it, and we won't consider the case where $x$ and $y$ span a light-like plane. 
There exists other natural ways  to define  these angles  but all the definitions are equivalent, up to an affine transformation 
(maybe complex), see for example  \cite{Sch07,Cho09}. 

\begin{lemma}\label{lem: antisym}
 Let $\langle \cdot ,\cdot \rangle'$ be the bilinear form $-\langle \cdot ,\cdot \rangle_1$.
\begin{enumerate}
 \item If $\norm{x}_1\in \mathbb{R}_+$, then $\norm{x}'\in i\mathbb{R}_+$ and  $\norm{x}_1=i\norm{x}'$.
\item If $\norm{x}_1\in i\mathbb{R}_+$, then $\norm{x}'\in \mathbb{R}_+$ and  $\norm{x}_1=-i\norm{x}'$.
\end{enumerate}
\end{lemma}
\begin{proof}
If $\norm{x}_1\in \mathbb{R}_+$ (resp. $i\mathbb{R}_+$), then $\langle x,x\rangle_1>0$ (resp. $<0$), so $\langle x,x\rangle'<0$ (resp. $>0$) and $\norm{x}'\in i\mathbb{R}_+$ (resp. $\mathbb{R}_+$). 
It is immediate to check that in both cases $\norm{x}_1=\sqrt{-1}\norm{x}'$. The choice of the square root
of $-1$ follows because we know the spaces the terms belong to.
\end{proof}

\begin{lemma}\label{lem:ADSconvexite}
In a Lorentzian space-form $M$, consider an angle formed by two different space-like half-hyperplanes $S_1,S_2$ meeting at a codimension $2$ plane $E$, such that there exists a time-like half-hyperplane $T$ containing $E$. Let $\alpha_i$ be the dihedral angle between $S_i$  and $T$.
\begin{enumerate}
 \item The half-planes $S_1$ and $S_2$ belong to the same hyperplane if and only if $\sinh \alpha_1+\sinh \alpha_2=0$.
\item  If $S_1$ and $S_2$ are not in the same hyperplane, then $T$ is inside the convex side of the angle if and only if $\sinh \alpha_1+\sinh \alpha_2<0$.   
\end{enumerate}
\end{lemma}
In the lemma above and in all the remainder of the section, a dihedral angle is an interior dihedral angle:  for any 
$x\in E$, $\alpha_i$ is the angle in $T_xM$ between a vector of $T_xS_i$ orthogonal to $T_xE$ and a  vector  of $T_xT$ orthogonal to $T_xE$.
\begin{proof}
By considering the problem in the Minkowski plane spanned by $y_1$ and $y_2$ in $T_xM$, it is easy to see that the half-plane are in the same plane if and only if $y_1=-y_2$ and that the convexity is equivalent to $\langle x,-y_2 \rangle > \langle x, y_1 \rangle $, that gives the condition of the lemma by definition of $\sinh\alpha_i$. 
\end{proof}

\begin{lemma}\label{lem:trigoHS}
Consider a contractible triangle in $dS^2$, with two space-like edges of length $a$ and $c$ and a time-like edge of length $ib$ and angles $\alpha,i\beta,\gamma$ ($a,b,c,\alpha,\beta,\gamma$ are real numbers). Then
\begin{eqnarray*}
 \  &&\left\{
      \begin{matrix}
        \cos a=\cosh b \cos c + \sinh b \sin c \sinh \alpha \\
        \cosh b= \cos c \cos a + \sin c \sin a \cosh \beta \\
        \cos c = \cos a \cosh b+\sin a \sinh b \sinh \gamma \\
      \end{matrix}
    \right., \\
  \ &&\frac{\sin a }{\sinh \alpha}=\frac{\sinh b}{\sinh \beta}=\frac{\sin c}{\cosh \gamma}, \\
\ &&\cosh \beta=\sinh \alpha \sinh \gamma+\cosh \alpha \cosh \gamma \cosh b.
\end{eqnarray*}
\end{lemma} 
\begin{proof}
 The formulas are proved for example in \cite{Cho09}.
\end{proof}

\begin{lemma}
Consider a triangle in $AdS_3$ contained in a time-like totally geodesic hypersurface, with a space-like edge of length $b$ and two time-like edges of length $ia, ic$. Let $\alpha,\beta,\gamma$ be the corresponding angles ($a,b,c,\alpha,\beta,\gamma$ are real numbers). If we consider $\alpha$ as a function of $a,c,\beta$, then 
\begin{equation}\label{eq:derADS1}
 \frac{\partial \alpha}{\partial a}=\frac{\cosh \gamma}{\sinh b},
\end{equation}
\begin{equation}\label{eq:derADS2}
\frac{\partial \alpha}{\partial c}=-\frac{\cosh b\cosh \alpha}{\sinh b}.
\end{equation}

If we suppose moreover that the triangle is such that $a=c$, and if we deform it among isosceles triangles we obtain:
\begin{equation}\label{eq: tri iso}
 \frac{\partial \alpha}{\partial a}=\frac{\cosh \alpha(1-\cosh b)}{\sinh b}.
\end{equation}
\end{lemma} 

\begin{proof}
In the ambient vector space, the triangle is contained in the intersection of the quadric with an hyperplane of signature $(+,-,-)$, that is, orthogonal to a space-like vector. Without loss of generality 
we consider that this space-like vector is given by the first coordinate in $\mathbb{R}^{2,2}$. So the triangle is contained in 
$$\{(x_2,x_3,x_4)\in \mathbb{R}^3 \vert x_2^2-x_3^2-x_4^2=-1 \} $$
which  is anti-isometric to $dS^2$. 
It is clear that the corresponding triangle in $dS^2$ is contractible and from the definitions of the pseudo-distances   it has edge lengths $a,ib,c$. The angle between the two space-like edges of length $a$ and $c$ is the angle between two space-like vectors spanning a time-like plane and meeting the same component of the resulting geodesic on the corresponding de Sitter space (this is because the triangle is contractible and the edge opposite to the angle is time-like). Hence from Lemma~\ref{lem: antisym}, the angle is $i\beta$, and
it also follows that the two other angles are  $-\alpha$ and $-\gamma$.

Using formulas of Lemma~\ref{lem:trigoHS}, we compute  (\ref{eq:derADS1}) and (\ref{eq:derADS2})  similarly to the corresponding formulas of Lemma~\ref{lem:trigoS}.

Equation (\ref{eq: tri iso}) follows by adding (\ref{eq:derADS1}) and (\ref{eq:derADS2}) and equalize 
$a$ and $c$ (that implies $\alpha=\gamma$).
\end{proof}

\begin{lemma}\label{lem:trigoHS2}
Consider a  triangle in $HS^2$, with two vertices in $dS^2$ joined by a space-like edge of length $a$, and a vertex in $H^2$, with edges of length $b,c$. Denote the corresponding angles by $\alpha, \beta,\gamma$ (all the data are real numbers). Then 
\begin{eqnarray*}
\left\{
      \begin{matrix}
            \cos a = - \sinh b \sinh c+\cosh b \cosh c \cos \alpha \\
       \sinh b=\cos a \sinh c + \sin a \cosh c \sinh \beta  \\
     \sinh c=\cos a \sinh b + \sin a \cosh b \sinh \gamma \\  
  \end{matrix}
    \right..
\end{eqnarray*}
If $a$ is considered as a function of $b,c,\alpha$, then
\begin{equation}\label{eq:derADS3}
 \frac{\partial a}{\partial b}= \sinh \gamma.
\end{equation}
\end{lemma} 
\begin{proof}
 The cosine laws come from \cite{Cho09}. The  derivative is computed in the same way as (\ref{eq: trigo sph 2}).
\end{proof}

\subsubsection{Jacobian matrix.}

The faces of the Fuchsian polyhedral  surface $P$ are decomposed in triangles. If we extend each 
half-ray to the point $o$, the dual of $H$ antipodal to $H^*$, each triangle gives a pyramid with forth vertex $o$. 
The height $h_x$ from $H$ to the vertex on $r_x$ should become an height $h_x+\pi/2$ from $o$ to this vertex, but this will change nothing in the arguments below.
We consider the same notations as in the case of the sphere: these are the one of  Figure~\ref{fig:pyramid} --- 
up to change $h_x$ and $h_y$ by $ih_x$ and $ih_y$.
  The cone-angle $\omega_x$ at
the vertex $x$ lying on  the half-ray $r_x$ is
 decomposed as a sum of angles of the form $\omega_{xyz}$ (note that $\omega_{xyz}=\omega_{xzy}$). The link of the pyramid at $x$ is a triangle of the nature of the one of the statement of Lemma~\ref{lem:trigoHS2}, and in which $\omega_{xyz}$ is a function of $\rho_{xy}, \rho_{xz}$ and $d_{xyz}$  when a height varies. Note that 
$d_{xyz}$ remains  fixed by hypothesis.

We denote $a_{xy}:=\frac{\partial \omega_x}{\partial h_y}$.

\begin{enumerate}
 \item  If $x$ is not joined to a vertex belonging to $G y$, then $a_{xy}=0$.
\item Consider a vertex  $y\notin G x$, such that there is an edge between $x$ and $y$. 
 We denote by $s_1,\cdots,s_m$ the vertices joined by an edge to $x$, numbered in the direct order around $x$, and such that $y=s_0=s_{m+1}$. Then
\begin{eqnarray*}
 \  a_{xy}=\frac{\partial \omega_x}{\partial h_y}= 
\sum_{j=0}^m \frac{\partial \omega_{xs_js_{j+1}}}{\partial h_y} 
 = \sum_{j=0}^m \frac{\partial \omega_{xs_js_{j+1}}}{\partial \rho_{xs_j}}\frac{\partial \rho_{xs_j}}{\partial h_y}
+\frac{\partial \omega_{xs_js_{j+1}}}{\partial \rho_{xs_{j+1}}}\frac{\partial \rho_{xs_{j+1}}}{\partial h_y} \\
\ =  \sum_{j=0}^m \left(\frac{\partial \omega_{xs_{j-1}s_{j}}}{\partial \rho_{xs_j}}
+\frac{\partial \omega_{xs_js_{j+1}}}{\partial \rho_{xs_j}}\right)\frac{\partial \rho_{xs_{j}}}{\partial h_y}=
\sum_{j=0}^m  a_{xy}^j.
\end{eqnarray*}
\begin{enumerate}
\item If  $s_j\notin G y$ then $a_{xy}^j=0$,
\item otherwise
\begin{eqnarray*}
 \ &&  a_{xy}^j\stackrel{(\ref{eq:derADS3}),(\ref{eq:derADS1})}{=}\left(\sinh\alpha_{xs_{j-1}s_j}+\sinh\alpha_{xs_js_{j+1}}\right)\frac{\cosh \rho_{s_jx}}{\sinh\ell_{xs_j}} 
\end{eqnarray*}
and it follows from Lemma~\ref{lem:ADSconvexite} that 
\begin{enumerate}
\item if the edge between $x$ and $s_j$ is a false edge then $a_{xy}^j=0$,
\item otherwise, as $P$ is convex,  $a_{xy}^j<0$.
\end{enumerate}
\end{enumerate}
\end{enumerate}

\begin{remark}
In general, $\rho_{s_jx}\not= \rho_{xs_j}$, then $a_{xy}^j\not= a_{yx}^j$, and so at a first sight  $a_{xy}\not= a_{yx}$. This should say that there is no functional $f$ from the set of heights to $\mathbb{R}$ such that 
the gradient of $f$ is equal to $C_{\Gamma}$. In the Euclidean case there is a  famous variational proof
of the Minkowski theorem, due to Minkowski himself. It is not clear if a variational proof could be worked out in the AdS case.
\end{remark}

Now we compute the diagonal terms $a_{xx}$. 

\begin{eqnarray*}
 \  a_{xx}=\frac{\partial \omega_x}{\partial h_x}= 
\sum_{j=0}^m \frac{\partial \omega_{xs_js_{j+1}}}{\partial h_x} 
 = \sum_{j=0}^m \frac{\partial \omega_{xs_js_{j+1}}}{\partial \rho_{xs_j}}\frac{\partial \rho_{xs_j}}{\partial h_x}
+\frac{\partial \omega_{xs_js_{j+1}}}{\partial \rho_{xs_{j+1}}}\frac{\partial \rho_{xs_{j+1}}}{\partial h_x} \\
\ =  \sum_{j=0}^m \left(\frac{\partial \omega_{xs_{j-1}s_{j}}}{\partial \rho_{xs_j}}
+\frac{\partial \omega_{xs_js_{j+1}}}{\partial \rho_{xs_j}}\right)\frac{\partial \rho_{xs_{j}}}{\partial h_x}=
\sum_{j=0}^m a_{xx}^j,
\end{eqnarray*}
and 
\begin{enumerate}
 \item if $s_j\in G y$ with $y\notin G x$ then
\begin{eqnarray*}
\  a_{xx}^j&=&\left(\frac{\partial \omega_{xs_{j-1}s_{j}}}{\partial \rho_{xs_j}}
+\frac{\partial \omega_{xs_js_{j+1}}}{\partial \rho_{xs_j}}\right)\frac{\partial \rho_{xs_{j}}}{\partial h_x}\\
\ &\stackrel{(\ref{eq:derADS1}),(\ref{eq:derADS2})}{=}& \left(\frac{\partial \omega_{xs_{j-1}s_{j}}}{\partial \rho_{xs_j}}
+\frac{\partial \omega_{xs_js_{j+1}}}{\partial \rho_{xs_j}}\right)
\left(-\cosh \ell_{xy}\frac{\cosh \rho_{xy}}{\cosh \rho_{yx}} \right)\frac{\partial \rho_{xs_{j}}}{\partial h_y}
= -\cosh \ell_{xy} a_{yx}^j,
\end{eqnarray*}
which is a non-negative number,
\item otherwise $s_j\in G x$ and 
\begin{eqnarray*}\label{eq:diag}
a_{xx}^j=\left(\frac{\partial \omega_{xs_{j-1}s_{j}}}{\partial \rho_{xs_j}}
+\frac{\partial \omega_{xs_js_{j+1}}}{\partial \rho_{xs_j}}\right)\frac{\partial \rho_{xs_{j}}}{\partial h_x}
\stackrel{(\ref{eq: tri iso}),(\ref{eq:derADS3})}{=}\left(\sinh\alpha_{xs_{j-1}s_j}+\sinh\alpha_{xs_js_{j+1}}\right)
\frac{\cosh \rho_{xs_j}(1-\cosh \ell_{xs_j})}{\sinh \ell_{xs_j}}. 
\end{eqnarray*}
which is a non-negative number.
\end{enumerate}

Let us note that $a_{xx}>0$. There are two cases.
\begin{itemize}
 \item $P$ is made of the orbit of only one vertex. In this case there must exists a true edge between $x$ and another vertex in the orbit of $x$, 
and then the formula above says that $a_{xx}>0$.
\item $P$ is made of the orbit of at least two vertices. So there must be a true edge between two vertices $x$ and $y$ lying in different orbits. 
In this case we have seen that $a_{xy}^j<0$, hence $a_{xx}^j=-\cosh \ell_{xy} a^j_{yx}>0$ and the result follows as all the other 
$a_{xy}^j$ are non-positive.
\end{itemize}

It follows that  
$$|a_{xx}|>\sum_{y\not=x} |a_{yx}|, $$
 that means that the matrix $(a_{xy})_{xy}$ is strictly diagonally dominant, and then invertible by an elementary result of linear algebra. Moreover  $C_{\Gamma}$ is clearly $C^1$. 
This completes the proof of Lemma~\ref{lem : jac}.

\bibliographystyle{alpha}
\bibliography{fs}

\end{document}